\documentclass[11pt]{amsart}
\usepackage[paperheight=11in, 
    paperwidth=8.5in,
    outer=1in,
    inner=1in,
    bottom=1in,
    top=1in]{geometry}

\usepackage[all]{xy}
\usepackage{amsmath, amsthm, amsfonts, amssymb, mathrsfs}
\usepackage{graphicx}
\usepackage{url}
\usepackage{color}
\usepackage{bbm}
\usepackage{tikz-cd}

\usepackage[T1]{fontenc}
\usepackage{enumitem}

\usepackage[hidelinks]{hyperref} 
 \hypersetup{
    colorlinks,
    citecolor=magenta,
    filecolor=magenta,
    linkcolor=blue,
    urlcolor=black
}

\usepackage{cleveref}

\usepackage[backend=bibtex, style=alphabetic, maxnames=99]{biblatex}
\bibliography{BES_SimpGenChowMat_JEMS}
\AtBeginBibliography{\footnotesize}

\newcommand{\QQ}{\mathbb Q}
\newcommand{\RR}{\mathbb R}

\newcommand{\PP}{\mathbb P}
\newcommand{\ZZ}{\mathbb Z}
\newcommand{\kk}{\mathbbm k}

\theoremstyle{remark}
\newtheorem*{claim}{Claim}
\newenvironment{claimproof}{\par\noindent\textit{Proof of claim.}\space}{\hfill$\diamond$\medskip\par}
\theoremstyle{definition}
\newtheorem{cor}[subsubsection]{Corollary}
\newtheorem{lem}[subsubsection]{Lemma}
\newtheorem{prop}[subsubsection]{Proposition}
\newtheorem{thm}[subsubsection]{Theorem}
\newtheorem{defn}[subsubsection]{Definition}
\newtheorem{eg}[subsubsection]{Example}
\newtheorem{rem}[subsubsection]{Remark}

\DeclareMathOperator{\rk}{rk}
\newcommand{\myChow}{A_\nabla}
\newcommand{\trop}{\operatorname{trop}}
\newcommand{\be}{\mathbf e}
\newcommand{\DHR}{\operatorname{DHR}}
\newcommand{\DHRtext}{Dragon Hall-Rado\ }

\title{Simplicial generation of Chow rings of matroids}

\author{Spencer Backman, Christopher Eur, Connor Simpson}

\address{The University of Vermont, Burlington, VT. USA.}
\email{sbackman@uvm.edu}

\address{Harvard University.  Cambridge, MA. USA.}
\email{ceur@math.harvard.edu}

\address{University of Wisconsin--Madison. Madison, WI. USA}
\email{csimpson6@wisc.edu}

\begin{document}

\maketitle

\begin{abstract}
We introduce a presentation of the Chow ring of a matroid by a new set of generators, called ``simplicial generators.''  These generators are analogous to nef divisors on projective toric varieties, and admit a combinatorial interpretation via the theory of matroid quotients.  Using this combinatorial interpretation, we (i) produce a bijection between a monomial basis of the Chow ring and a relative generalization of Schubert matroids, (ii)  recover the Poincar\'e duality property, (iii) give a formula for the volume polynomial, which we show is log-concave in the positive orthant, and (iv) recover the validity of Hodge-Riemann relations in degree 1, which is the part of the Hodge theory of matroids that currently accounts for all combinatorial applications of \cite{AHK18}.
Our work avoids the use of ``flips,'' the key technical tool employed in \cite{AHK18}.
\end{abstract}


\section{Introduction}

Chow rings of matroids were introduced in \cite{FY04} as generalizations of cohomology rings of wonderful compactifications of hyperplane arrangement complements \cite{dCP95}:

\begin{defn}
Let $M$ be a loopless matroid of rank $r=d+1$ on a ground set $E$.  The \textbf{Chow ring} of $M$ is a graded ring $A_{FY}^\bullet(M) = \bigoplus_{i=0}^d A_{FY}^i(M)$ defined as
$$A_{FY}^\bullet(M):= \frac{\RR[z_F : F\subseteq E \textnormal{ a nonempty flat of $M$}]}{\langle z_Fz_{F'} \ | \ F,F' \textnormal{ incomparable}\rangle + \langle \sum_{F\supseteq a} z_F  \ | \ a \textnormal{ an atom in the lattice of flats of $M$} \rangle}.$$
\end{defn}

\bigskip
We introduce a presentation of the Chow ring by a new set of generators:



\newtheorem*{dfn:simplicial}{Definition \ref{dfn:simplicial}}
\begin{dfn:simplicial}
Let $M$ be a loopless matroid on $E$.  The \textbf{simplicial presentation} $\myChow^\bullet(M)$ of the Chow ring of $M$ is the quotient of a polynomial ring $\RR[h_F \ | \ F \subseteq E \textnormal{ nonempty flat of $M$}]$ by the kernel of the surjective map
\[
\RR[h_F \ | \ F \subseteq E \textnormal{ nonempty flat of $M$}] \to A_{FY}^\bullet(M) \quad \textnormal{where}\quad  h_F \mapsto -\sum_{G \supseteq F} z_G.
\]
The generators $h_F$ of $\myChow^\bullet(M)$ are called the \textbf{simplicial generators} of the Chow ring of $M$.
\end{dfn:simplicial}

The simplicial presentation allows us to
\begin{enumerate}[label=(\roman*)]
\item interpret  the simplicial generators and their monomials in $\myChow^\bullet(M)$ combinatorially via the theory of matroid quotients;
\item recover the Poincar\'e duality property of $\myChow^\bullet(M)$;
\item compute the volume polynomial of $\myChow^\bullet(M)$, which we show to be Lorentzian in the sense of \cite{BH20}, and whose formula generalizes a main result of \cite{Pos09}; and
\item give a simplified proof of the Hodge-Riemann relations in degree 0 and 1 for matroids, which is the part of the Hodge theory of matroids in \cite{AHK18} that accounts for all currently known combinatorial applications of \cite{AHK18}.
\end{enumerate}

We achieve these by establishing combinatorial properties of the simplicial presentation that reflect the following two geometric observations:

\begin{enumerate}[label = (\Alph*)]
\item\label{obs:transport} Suppose base-point-free divisors $D_1, \ldots, D_s$ on a variety $X$ generate the Chow ring $A^\bullet(X)$, and $Y\subset X$ is a subvariety.  Then the pullback map  $A^\bullet(X)\to A^\bullet(Y)$ can be computed via the intersections of $Y$ with general hyperplane pullbacks of $X \to \PP(H^0(D_i))$.  If furthermore the pullback map is surjective, then these intersections generate $A^\bullet(Y)$.
\item\label{obs:hodgenef} For base-point-free divisors $D_1, \ldots, D_s$ whose sum is an ample divisor on a projective variety $Y$, let $\operatorname{vol}(t_1, \ldots, t_s)$ be the volume polynomial $\int_Y (t_1D_1 + \cdots + t_sD_s)^{\dim Y}$.  Then $\operatorname{vol}(\underline t)$, as a function $\RR^s \to \RR$, is positive and log-concave on the positive orthant $\RR_{>0}^s$ \cite[\S1.6.A]{Laz04}.
\end{enumerate}

When a matroid $M$ is realizable, its Chow ring is isomorphic to the Chow ring of a projective variety $Y_{\mathscr R(M)}$, which is a subvariety of a projective toric variety $X_{A_n}$, and the pullback $A^\bullet(X_{A_n}) \to A^\bullet(Y_{\mathscr R(M)})$ is surjective (see \S\ref{subsection:wondcpt}).  In this case, unlike the classical presentation $A^\bullet_{FY}(M)$, the generators of our presentation $\myChow^\bullet(M)$ represent \emph{base-point-free} divisor classes on $Y_{\mathscr R(M)}$, obtained as pullbacks of base-point-free divisor classes on $X_{A_n}$ that generate $A^\bullet(X_{A_n})$ (see \S\ref{subsection:nefpresentation}).  We demonstrate that even when $M$ is not necessarily realizable, the generators of $\myChow^\bullet(M)$ display behaviors analogous to those of base-point-free divisors on the projective varieties.

\subsection{Main results}

\subsubsection*{\textnormal{(i)} Simplicial generators and their monomials as matroid quotients}

The \textbf{principal truncation} of a matroid $M$ by a flat $F$ is a new matroid with bases
$
\{B\setminus f \colon B \textnormal{ a basis of }M \textnormal{ and } f\in B\cap F \neq \emptyset\}.
$
When $M$ is realizable, so that $h_F \in \myChow^1(M)$ represents a base-point-free divisor class on a projective variety $Y_{\mathscr R(M)}$, the principal truncation by $F$ corresponds to any general hyperplane pullback of the map defined by $h_F$ (Remarks \ref{rem:geomhyper} \& \ref{rem:simplicialpullback}).
For an arbitrary matroid, the following analogue holds.

\newtheorem*{thm:hyper}{Theorem \ref{thm:hyper}}
\begin{thm:hyper}
Let $M$ be a loopless matroid.  The simplicial generator $h_F\in \myChow^\bullet(M)$ correspond, via the cap product (see \S\ref{subsection:ChowMW} for a definition), to the principal truncation of $M$ by the flat $F$.
\end{thm:hyper}

We use this interpretation of the simplicial generators to study their monomials in $\myChow^\bullet(M)$.  A monomial basis for $\myChow^\bullet(M)$ is obtained by carrying over a Gr\"obner basis computation for the classical presentation $A_{FY}^\bullet(M)$ in \cite{FY04} (Proposition \ref{prop:gbH} \& Corollary \ref{cor:monombasis}).  This monomial basis, which we call the \textbf{nested basis}, now has the following combinatorial interpretation.

\newtheorem*{dfn:relnestedquot}{Definition \ref{dfn:relnestedquot}}
\begin{dfn:relnestedquot}
Let $M'$ be a matroid quotient of $M$ (that is, every flat of $M'$ is a flat of $M$), denoted $f: M'\twoheadleftarrow M$.  Write $n_f(S) := \operatorname{rk}_M(S) - \operatorname{rk}_{M'}(S)$ for a subset $S$ of the ground set.  An \textbf{$f$-cyclic flat} of $f$ is a flat $F$ of $M'$ such that $F$ is minimal (with respect to inclusion) among the flats $G$ of $M'$ with $n_f(G) = n_f(F)$.  The matroid $M'$ is a \textbf{relative nested quotient} of $M$ if the $f$-cyclic flats form a chain.
\end{dfn:relnestedquot}

\newtheorem*{thm:monomnested}{Theorem \ref{thm:monomnested}}
\begin{thm:monomnested}
There is a natural bijection, via the cap product, between the nested basis of $\myChow^\bullet(M)$ and the set of relative nested quotients of $M$.
\end{thm:monomnested}

When the matroid $M$ is a Boolean matroid, i.e.\ a matroid whose ground set is a basis, the relative nested quotients of $M$ are known as ``nested matroids,'' studied previously in the context of Chow rings of permutohedral varieties in \cite{Ham17}.

\subsubsection*{\textnormal{(ii)} The Poincar\'e duality property} We build upon the bijection in Theorem \ref{thm:monomnested} to 
establish the following theorem, which mirrors the fact that for certain classes of varieties,\footnote{These include varieties on which rational and numerical equivalence coincide, which is a feature shared by all varieties that inspire the combinatorics of this paper.  See the first footnote in \S\ref{subsection:ChowMW}.} if $A^\bullet(X) \to A^\bullet(Y)$ is surjective for $Y\subset X$, then $A^\bullet(Y) \simeq A^\bullet(X) / \operatorname{ann}([Y])$ where $\operatorname{ann}([Y]) = \{\xi\in A^\bullet(X) \mid \xi \cdot [Y] = 0\}$.

\newtheorem*{thm:transport}{Theorem \ref{thm:transport}}
\begin{thm:transport}
Let $M$ be a loopless matroid on $\{0,1, \ldots, n\}$, and $X_{A_n}$ the permutohedral variety of dimension $n$.  Let $\Delta_M$ be the Bergman class of the matroid $M$ considered as an element of $A^\bullet(X_{A_n})$.  Then $A^\bullet(M) \simeq A^\bullet(X_{A_n}) / \operatorname{ann}(\Delta_M)$, where $\operatorname{ann}(\Delta_M) = \{\xi \in A^\bullet(X_{A_n}) \mid \xi \cdot \Delta_M = 0\}.$
\end{thm:transport}

The Poincar\'e duality property for $\myChow^\bullet(M)$ was established in \cite[Theorem 6.19]{AHK18} by a double induction.  In our case, the Poincar\'e duality property is a simple algebraic consequence of \Cref{thm:transport} (\Cref{cor:poincare}).

\subsubsection*{\textnormal{(iii)} The volume polynomial and its log-concavity} The following formula computes the intersection numbers of the simplicial generators $h_F$.

\newtheorem*{thm:dhr}{Theorem \ref{thm:dhr}}
\begin{thm:dhr}
For $M$ a loopless matroid of rank $r=d+1$, let $\int_M: A^d(M) \to \RR$ be the degree map.  For a multiset of nonempty flats $\{F_1, \ldots, F_d\}$, we have
$$\int_M h_{F_1} \cdots h_{F_d} = 
\begin{cases} 
1 & \textnormal{if $\operatorname{rk}_M(\bigcup_{j\in J} F_j) \geq |J| +1$ for every $\emptyset \subsetneq J\subseteq \{1, \ldots, d\}$}\\
0 & \textnormal{otherwise}.
\end{cases}$$
\end{thm:dhr}

%

That these intersection numbers are either 1 or 0 stands in stark contrast to calculations in the classical presentation of the Chow ring of a matroid, which were carried out in \cite{Eur20}.

\medskip
The intersection numbers collect together into the ``volume polynomial'' of $\myChow^\bullet(M)$, defined in the following corollary.  The corollary recovers \cite[Corollary 9.4]{Pos09} on volumes of generalized permutohedra when $M$ is set to be the Boolean matroid.

\newtheorem*{cor:volpolH}{Corollary \ref{cor:volpolH}}
\begin{cor:volpolH}
Let $M$ be a loopless matroid on $E$ of rank $d+1$.  The volume polynomial $VP_M^\nabla(\underline t)\in \QQ[t_F \ | \ F \subseteq E \textnormal{ nonempty flat in $M$}]$ of $\myChow^\bullet(M)$, defined as $VP_M^\nabla(\underline t) := \int_M \big(\sum_F t_F h_F\big)^d$, is
$$VP_M^\nabla(\underline t) = \sum_{(F_1, \ldots, F_d)} t_{F_1}\cdots t_{F_d}$$
where the sum is over ordered collections of flats $(F_1, \ldots, F_d)$ satisfying $\operatorname{rk}_M(\bigcup_{j\in J} F_j) \geq |J| +1$ for every $\emptyset \subsetneq J\subseteq \{1, \ldots, d\}$.
\end{cor:volpolH}

The volume polynomial $VP_M^\nabla$, as a real-valued function, is therefore positive in its positive orthant, as are volume polynomials of base-point-free divisors on projective varieties.  Moreover, it displays the log-concavity behavior described in \ref{obs:hodgenef}.


\newtheorem*{cor:VPlogconc}{\Cref{cor:VPlogconc}}
\begin{cor:VPlogconc}
As a real-valued function, the volume polynomial $VP_M^\nabla(\underline t)$ is log-concave in its positive orthant.
\end{cor:VPlogconc}
Log-concavity is a consequence of the following statement.
\newtheorem*{thm:VPLorentzian}{\Cref{thm:VPLorentzian}}
\begin{thm:VPLorentzian}
The volume polynomial $VP_M^\nabla(\underline t)$ of a loopless matroid $M$ is Lorentzian in the sense of \cite{BH20}.  
\end{thm:VPLorentzian}

Lorentzian polynomials are multivariate polynomials characterized by two conditions: one on their supports and one on their partial derivatives (see \S\ref{subsection:Lorentzian}).  We show that $VP_M^\nabla$ satisfies these two conditions by using \Cref{thm:dhr} to understand the support of $VP_M^\nabla$ (\Cref{prop:mconvex}) and by using \Cref{thm:hyper} to understand the partial derivatives of $VP_M^\nabla$ as volume polynomials of principal truncations of $M$.

\subsubsection*{\textnormal{(iv)} A simplified proof of the Hodge-Riemann relations in degree 1}

We use log-concavity of $VP_M^\nabla$ to give a simplified proof of the Hodge-Riemann relations in degree 1 for Chow rings of matroids:

%

\newtheorem*{thm:HRone}{Theorem \ref{thm:HRone}}
\begin{thm:HRone}
Let $M$ be a loopless matroid of rank $r = d+1$ on a ground set $E$, and $\ell\in A^1(M)$ a combinatorially ample divisor class (see \S\ref{subsection:chowring}).  Then the Chow ring $A^\bullet(M)$ with the degree map $\int_M$ satisfies the K\"ahler package in degree zero and one.  That is, for $i\leq 1$,
\begin{itemize}
\item[($\operatorname{HL}^{\leq1}$)] (hard Lefschetz in degree $\leq1$) the multiplication by $\ell$
$$L_\ell^i : A^i(M) \to A^{d-i}(M),\ a \mapsto \ell^{d-2i}a$$
is an isomorphism, and
\item[($\operatorname{HR}^{\leq1}$)] (Hodge-Riemann relations in degree $\leq1$) the symmetric form
$$(-1)^iQ_\ell^i: A^i(M) \times A^i(M) \to \RR, \ (x,y) \mapsto (-1)^i\int_M xy\ell^{d-2i}$$
is non-degenerate on $A^i(M)$ and positive-definite when restricted to the degree $i$ primitive space $P_\ell^i := \{z\in A^i(M) : \ell^{d-2i+1}z = 0\}$ of $\ell$.
\end{itemize}
\end{thm:HRone}

The authors of \cite{AHK18} establish the Hodge-Riemann relations in all degrees by a double-inductive argument that utilizes a generalization of Bergman fans of matroids.
In our case, the log-concavity of $VP_M^\nabla$ (\Cref{cor:VPlogconc}) provides us a key step in the induction in degrees $\leq 1$, and thus our proof of \Cref{thm:HRone} is a single induction on ranks of matroids and involves only the classical Bergman fans of matroids.

\medskip
The validity of the Hodge-Riemann relation in degree 1 accounts for all currently known combinatorial applications of \cite{AHK18}, including the Heron-Rota-Welsh conjecture.  Huh has posed the discovery of combinatorial applications for the Hodge-Riemann relations in higher degrees as an open problem \cite{Huh18a}.

\subsection*{Organization}\label{subsec:structure}
In \Cref{section:prelim}, we develop the necessary background on Chow rings of matroids.  In \Cref{section:nefpresentation}, we introduce the simplicial presentation of the Chow ring, study its structural properties via principal truncations, and introduce the relative nested quotients.  In Section \ref{section:poincare}, we establish the Poincar\'e duality property for matroids.  In Section \ref{section:volumepolynomial}, we calculate intersection numbers with respect to the simplicial presentation, and prove that the volume polynomial of the Chow ring in the simplicial presentation is Lorentzian.  Finally, in Section \ref{section:HIT} we apply our results to give a simplified proof of the Hodge theory in degree 1 for matroids.

\subsection*{Acknowledgements}
We are grateful to Vic Reiner and Federico Castillo for introducing the authors.  The second author thanks June Huh for helpful conversations about Lorentzian polynomials.  The first author was supported by a Zuckerman Postdoctoral Scholarship.  We thank Alex Fink for suggesting some minor corrections, and we thank Matt Larson for pointing out a minor error in Proposition~\ref{prop:DHRmatching} and in the proof of Proposition~\ref{prop:linindep}.

\section{Preliminaries}\label{section:prelim}

In this section, we review relevant background materials on Chow rings of matroids.  Familiarity with toric and tropical geometry, from which the combinatorial notions here originate, can be helpful but not necessary.  As references we point to \cite{Ful93} and \cite{CLS11} for toric geometry, and to \cite{FS97}, \cite[Chapter 6]{MS15}, and \cite[\S 4--\S 5]{AHK18} for tropical geometry.

\smallskip
In \S\ref{subsection:ChowMW}, we describe Chow cohomology rings and Minkowski weights of fans, and in \S\ref{subsection:chowring} we illustrate these notions in the setting of matroids.  These first two subsections are purely combinatorial.
In \S\ref{subsection:wondcpt}, we provide the underlying geometric picture that motivates many of the combinatorial constructions.

\subsection{Chow cohomology rings and Minkowski weights}\label{subsection:ChowMW}

We give a brief account of Chow cohomology rings and Minkowski weights of smooth fans, which are combinatorial analogues of cohomology rings and homology classes of algebraic varieties\footnote{We use real coefficients for Chow cohomology rings and Minkowski weights, although Chow rings of algebraic varieties initially take integral coefficients.  The algebraic varieties that motivate the constructions here---smooth complete toric varieties and wonderful compactifications---share the feature that the Chow ring, the integral cohomology ring, and the ring of algebraic cycles modulo numerical equivalence all coincide \cite[Appendix C.3.4]{EH16}.  In particular, their Chow rings are torsion-free.  In this paper, while most of our arguments work over $\ZZ$, we will always work over $\RR$ for convenience.}.

\medskip
We set the following notation and definitions for rational fans over a lattice.
\begin{itemize}
\item Let $N$ be a lattice of rank $n$, and $N^\vee$ the dual lattice. We write $N_\RR := N\otimes_\ZZ \RR$.
\item For $\Sigma\subset N_\RR$ a rational fan, let $\Sigma(k)$ be the set of $k$-dimensional cones of $\Sigma$.
\item For a ray $\rho \in \Sigma(1)$, write $u_\rho \in N$ for the \textbf{primitive ray vector} that generates $\rho \cap N$.
\item A fan $\Sigma$ is \textbf{smooth} if, for all cones $\sigma$ of $\Sigma$, the set of primitive ray vectors of $\sigma$ can be extended to a basis of $N$.  A smooth fan is \textbf{simplicial} in that every $k$-dimensional cone is generated by $k$ rays.
\item A fan $\Sigma$ is said to be \textbf{complete} if its support $|\Sigma|$ is equal to $N_\RR$.
\end{itemize}

\proof[Convention] Throughout this section, we assume that $\Sigma \subset N_\RR$ is a smooth fan of dimension $d$, which is not necessarily complete.

\begin{defn}
The \textbf{Chow cohomology ring}
$A^\bullet(\Sigma)$ of $\Sigma$ is a graded $\RR$-algebra
\[
A^\bullet(\Sigma) := \frac{\RR[x_\rho: \rho\in \Sigma(1)]}{\Big\langle \prod_{\rho\in S} x_\rho \ \left| \ \textnormal{$S\subseteq \Sigma(1)$ do not form a cone in $\Sigma$}\Big\rangle\right. + \Big\langle \sum_{\rho}  m(u_\rho) x_\rho \ \left| \ m \in N^\vee \Big\rangle\right.}.
\]
\end{defn}


Geometrically, the ring $A^\bullet(\Sigma)$ is the Chow ring $A^\bullet(X_\Sigma)$ of the toric variety $X_\Sigma$ associated to the fan $\Sigma$.  See \cite[\S10.1]{Dan78} for the case where $\Sigma$ is complete, and \cite{BDCP90} or \cite{Bri96} for the general case.  From this geometric description of $A^\bullet(\Sigma)$, or directly from the algebraic definition above, one can check that $A^\ell(\Sigma) = 0$ unless $0\leq \ell \leq d$.

\medskip
We call a linear combination $\sum_{\rho} c_\rho x_\rho \in \RR[x_\rho: \rho \in \Sigma(1)]$ of the variables $x_\rho$ a \textbf{divisor} on $\Sigma$.  Divisors of special interest in algebraic geometry are ample and nef divisors\footnote{Base-point-free divisor classes are nef, and on toric varieties nef conversely implies base-point-free \cite[Theorem 6.3.12]{CLS11}.  In agreement with the terminology of \cite{AHK18}, we will call certain divisors ``combinatorially nef'' (see \S\ref{subsection:chowring}) although they are furthermore base-point-free on wonderful compactifications of realizable matroids.}.  They have the following combinatorial description for a complete fan $\Sigma$ (i.e.\ a complete toric variety $X_\Sigma$).

\smallskip
A divisor $D = \sum_{\rho \in \Sigma(1)} c_\rho x_\rho$ on a complete fan $\Sigma$ defines a piecewise-linear function $\varphi_D: N_\RR \to \RR$, determined by being linear on each cone of $\Sigma$ with $\varphi_D(u_\rho) = c_\rho$.  We say that $D$ is a \textbf{nef divisor} if $\varphi_D$ is a convex function on $N_\RR$, that is, $\varphi_D(u) + \varphi_D(u') \geq \varphi(u+u')$ for all $u,u'\in N_\RR$.  If further the inequalities $\varphi_D(u) + \varphi_D(u') \geq \varphi(u+u')$ are strict whenever $u$ and $u'$ are not in a common cone of $\Sigma$, we say that $D$ is \textbf{ample}.
Nef (resp.\ ample) divisors on $\Sigma$ correspond to polytopes in $N_\RR^\vee$ whose outer normal fans coarsen (resp.\ equal) $\Sigma$.

\begin{thm}\cite[Theorems 6.1.5--6.1.7]{CLS11}\label{thm:nefdef}
Let $\Sigma$ be a smooth complete fan.  A nef divisor $D = \sum_{\rho \in \Sigma(1)} c_\rho x_\rho$ on $\Sigma$ defines a polytope $P_D \subset N_\RR^\vee$ by
$$P_D := \{m \in N^\vee_\RR \ | \ m(u_\rho) \leq c_\rho \ \forall \rho \in \Sigma(1)\},$$
whose outer normal fan coarsens $\Sigma$.  Conversely, such polytope $P \subset N^\vee_\RR$ defines a nef divisor
\[
D_P := \sum_{\rho \in \Sigma(1)} \max\{m(u_\rho) \ | \ m\in P\} x_\rho.
\]
A nef divisor $D$ is ample if the outer normal fan of $P_D$ is equal to $\Sigma$.
%
%
\end{thm}

A divisor $D$ defines an element $[D] \in A^1(\Sigma)$, which we call the \textbf{divisor class} (of $D$) on $\Sigma$.
We say that a divisor class $\zeta\in A^1(\Sigma)$ is \textbf{nef} (resp.\ \textbf{ample}) if any choice of a divisor $D$ representing $\zeta$ is nef (resp.\ ample).  This is well-defined because two divisors $D$ and $D'$ represent the same divisor class if and only if $\varphi_D - \varphi_{D'}$ is a linear function on $N_\RR$.  In terms of polytopes, two nef divisors $D$ and $D'$ define the same divisor class in if and only if $P_D$ and $P_{D'}$ are parallel translates.  

\begin{rem}\label{rem:nefeff}
Any nef divisor class $[D]\in A^1(\Sigma)$ is \textbf{effective}; that is, it can be written as non-negative linear combination $D = \sum_{\rho\in \Sigma(1)} c_\rho x_\rho$ (with $c_\rho \geq 0 \ \forall \rho \in \Sigma(1)$).  This is an immediate consequence of Theorem \ref{thm:nefdef}:  Given a nef divisor $D$, translating if necessary one can assume that the polytope $P_D$ contains the origin in its relative interior.
\end{rem}

With $A^\bullet(\Sigma)$ as an analogue of a cohomology ring, we describe  analogues of homology groups.

\begin{defn}
An $\ell$-dimensional \textbf{Minkowski weight} on $\Sigma$ is a function $\Delta: \Sigma(\ell) \to \RR$ such that for each $\tau\in \Sigma(\ell-1)$, the function $\Delta$ satisfies the balancing condition
\[
\sum_{\substack{\sigma \supset \tau \\ \sigma \in\Sigma(\ell)}} \Delta(\sigma) u_{\sigma\setminus\tau} \in \operatorname{span}_\RR(\tau),
\]
where ${\sigma\setminus\tau}$ denotes the unique ray of $\sigma$ that is not in $\tau$.  The \textbf{support} of $\Delta$, denoted $|\Delta|$, is the union of cones $\sigma\in \Sigma(\ell)$ such that $\Delta(\sigma) \neq 0$.  We write $\operatorname{MW}_\ell(\Sigma)$ for the group (under addition) of $\ell$-dimensional Minkowski weights on $\Sigma$.
\end{defn}

The groups of Minkowski weights are analogues of homology groups because they are dual to the Chow cohomology ring in the following way.

\begin{lem}\label{lem:kronecker}\cite[Theorem 6.7.5]{MS15}\footnote{Currently \cite[Theorem 6.7.5]{MS15} has a typo---it is missing $\operatorname{Hom}(\cdot, \ZZ)$.  The statement here was made implicitly in \cite{FMSS95}, and follows the notation of \cite[Proposition 5.6]{AHK18}.} For $0 \leq \ell \leq d$, we have an isomorphism
\[
t_\Sigma: \operatorname{MW}_\ell(\Sigma) \overset\sim\to \operatorname{Hom}(A^\ell(\Sigma),\ZZ), \quad \textnormal{determined by } \Delta\mapsto \Big( (\textstyle \prod_{\rho\in \sigma(1)}x_\rho) \mapsto \Delta(\sigma)\Big).
\]
\end{lem}

This isomorphism is an analogue of the Kronecker duality map in algebraic topology.  We use it to define combinatorial analogues of some standard operations in algebraic topology.  
We define the \textbf{cap product} by
$$A^k(\Sigma) \times \operatorname{MW}_\ell(\Sigma) \to \operatorname{MW}_{\ell-k}(\Sigma), \quad (\xi, \Delta) \mapsto \xi \cap \Delta := \Big(\sigma \mapsto (t_\Sigma\Delta)(\xi\cdot \textstyle \prod_{\rho \in \sigma(1)}x_\rho)\Big),$$
which makes $\operatorname{MW}_\bullet(\Sigma)$ into a graded $A^\bullet(\Sigma)$-module.
When $\Sigma$ satisfies $\operatorname{MW}_d(\Sigma) \simeq \RR$, the \textbf{fundamental class} $\Delta_\Sigma$ is defined as its generator (unique up to scaling), and the cap product with the fundamental class defines the map
$$\delta_\Sigma: A^\bullet(\Sigma) \to \operatorname{MW}_{d-\bullet}(\Sigma), \quad \xi\mapsto \xi\cap\Delta_\Sigma.$$
In particular, noting that $\operatorname{MW}_0(\Sigma) = \RR$, the \textbf{degree map}
is defined as
\[
\int_\Sigma: A^d(\Sigma) \to \RR, \quad \xi\mapsto \xi\cap \Delta_\Sigma.
\]

If $\Sigma$ is complete, one can check that $\operatorname{MW}_n(\Sigma) \simeq \RR$, where the fundamental class $\Delta_\Sigma$ is given by $\Delta_\Sigma(\sigma) = 1$ for each cone $\sigma \in \Sigma(n)$.  In this case, we have the following analogue of the Poincar\'e duality theorem in algebraic topology.

\begin{thm}\label{thm:FS}\cite[Theorem 3.1, Proposition 4.1.(b), Theorem 4.2]{FS97}
For $\Sigma$ a smooth complete fan, the cap product with the fundamental class $\Delta_\Sigma$
\[
\delta_\Sigma: A^k(\Sigma) \overset\sim\to \operatorname{MW}_{n-k}(\Sigma), \quad \xi \mapsto \xi \cap \Delta_\Sigma
\]
is an isomorphism for each $0\leq k\leq n$.
Equivalently (by \Cref{lem:kronecker}), the pairing
\[
A^k(\Sigma) \times A^{n-k}(\Sigma) \to \RR, \quad (a, a') \mapsto \int_\Sigma a a'
\]
is non-degenerate for each $0 \leq k \leq n$.
\end{thm}

The isomorphisms in \Cref{thm:FS} make $\bigoplus_{i=0}^n \operatorname{MW}_{n-i}(\Sigma)$ into a graded ring when $\Sigma$ is complete.
We write $\operatorname{MW}^\bullet(\Sigma) := \operatorname{MW}_{n-\bullet}(\Sigma)$ for this graded ring.
For complete fans, the resulting multiplication structure on the Minkowski weights is known as the \textbf{stable intersection}, denoted $\cap_{st}$, in tropical geometry.  We will only need a special case of stable intersections\footnote{See \cite{FS97}, \cite[\S3.6]{MS15}, or \cite{JY16} for a definition of stable intersections.  It may help to note the suggestiveness of the notations here---we have $\xi \cap \Delta = \delta_\Sigma(\xi) \cap_{st} \Delta$.}, which we provide explicitly in the context of matroids in \S\ref{subsection:matquot} (\Cref{prop:matprod}).

\medskip
In the proof of \Cref{lem:hyper}, we will need the following explicit description of the map $\delta_\Sigma: A^1(\Sigma) \overset{\sim}\to \operatorname{MW}_{n-1}(\Sigma)$ for nef divisor classes on a complete fan $\Sigma$.  It is familiar to tropical geometers as tropical hypersurfaces \cite[Proposition 3.3.2 \& Theorem 6.7.7]{MS15}.

\begin{prop}\label{prop:trophyper}
Let $D$ be a nef divisor on $\Sigma$ such that the corresponding polytope $P_D$ is a lattice polytope.  Then the Minkowski weight $\Delta_{P_D} := \delta_\Sigma([D]) \in \operatorname{MW}_{n-1}(\Sigma)$ given by \Cref{thm:FS} is defined by
\[
\Delta_{P_D}(\tau) =
\begin{cases}
\ell(P_D(\sigma)) & \textnormal{if there is $\sigma\in \Sigma_{P_D}(n-1)$ such that $|\tau| \subseteq |\sigma|$}\\
0 & \textnormal{otherwise}
\end{cases}\qquad \textnormal{for each $\tau\in \Sigma(n-1)$},
\]
where $P_D(\sigma)$ is the edge of $P_D$ corresponding to the cone $\sigma \in \Sigma_Q(n-1)$, and $\ell(P_D(\sigma))$ is its lattice length, i.e.\ the number of lattice points on $P_D(\sigma)$ minus one.
\end{prop}

We end this subsection by noting the functoriality of the constructions here.  An inclusion of fans $\iota: \Sigma' \hookrightarrow \Sigma$ defines the \textbf{pullback map} $\iota^*$, which is a surjective map of graded rings
$$\iota^*: A^\bullet(\Sigma) \twoheadrightarrow A^\bullet({\Sigma'}), \quad x_\rho \mapsto
\begin{cases}
x_\rho & \textnormal{if $\rho\in \Sigma'(1)$}\\
0 & \textnormal{otherwise}.
\end{cases}$$
Comparing the presentations of $A^\bullet(\Sigma)$ and $A^\bullet(\Sigma')$, one checks easily that this map coincides with the quotient of $A^\bullet(\Sigma)$ by the ideal $\langle x_\rho \mid \rho \in \Sigma(1) \setminus \Sigma'(1)\rangle \subset A^\bullet(\Sigma)$.  Dually, a Minkowski weight $\Delta'$ on $\Sigma'$ is naturally a Minkowski weight on $\Sigma$.  In this case we often abuse the notation and write $\Delta'$ for both Minkowski weights.

\begin{rem}\label{rem:functorial}
Unraveling the definitions, one checks that the cap product is functorial in the following sense:  The pullback map $\iota^*: A^\bullet(\Sigma) \to A^\bullet(\Sigma')$ makes $\operatorname{MW}_\bullet(\Sigma')$ into a $A^\bullet(\Sigma)$-module.  Explicitly, if $\xi \in A^\bullet(\Sigma)$ and $\Delta'\in \operatorname{MW}_\ell(\Sigma')$, then $\iota^* \xi \cap \Delta' = \xi \cap \Delta'$, where $\Delta'$ on the right hand side is considered as a Minkowski weight on $\Sigma$.
\end{rem}

\subsection{Bergman classes and Chow rings of matroids} \label{subsection:chowring}

We now specialize our discussion to matroids.  We begin with the braid fan, on which matroids will arise as certain Minkowski weights.

\medskip
First, we fix some notations.  Let $E := \{0, 1, \ldots, n\}$, and for a subset $S\subseteq E$ write $\be_S := \sum_{i\in S} \be_i$, where $\be_0, \ldots, \be_n$ is the standard basis of $\ZZ^E$.  Let $N$ be the lattice $N = \ZZ^E/\ZZ\mathbf \be_E$, and write $u_S$ for the image of $\be_S$ in $N$.  The dual lattice of $N$ is $N^\vee = (\ZZ\be_E)^\perp = \{(y_0, \ldots, y_n) \in \ZZ^E \mid \sum_{i=0}^n y_i = 0\}$.  

\medskip
The \textbf{braid fan} (of dimension $n$), denoted $\Sigma_{A_n}$, is the outer normal fan of the standard \textbf{permutohedron} (of dimension $n$), which is the polytope
\[
\Pi_n := \operatorname{Conv}(w(0,1,\ldots, n) \in \RR^E \mid \textnormal{all permutations $w$ of $E$}).
\]
Concretely, the braid fan $\Sigma_{A_n}$ is a complete fan in $N_\RR$ whose cones are $\operatorname{Cone}(u_{S_1}, \ldots, u_{S_k}) \subset N_\RR$, one for each chain of nonempty proper subsets $\emptyset \subsetneq S_1 \subsetneq \cdots \subsetneq S_k \subsetneq E$.  In particular, the primitive rays of $\Sigma_{A_n}$ are $\{u_S \mid \emptyset \subsetneq S \subsetneq E\}$.
This fan is also known as the Coxeter complex of the type $A$ root system, hence the notation $\Sigma_{A_n}$.  The associated toric variety of $\Sigma_{A_n}$, denoted $X_{A_n}$, is often called the \textbf{permutohedral variety} (of dimension $n$).

\medskip
We assume familiarity with the basics of matroids, and refer to \cite{Wel76, Oxl11} as general references.  We fix the following notation for matroids:  We write $U_{r,E}$ for the uniform matroid of rank $r$ on $E$, and we set a matroid $M$ to have
\begin{itemize}
        \item ground set $E = \{0,1, \ldots, n\}$,
        \item $\mathcal B(M)$ the set of bases of $M$, 
        \item $\rk_M$ the rank function of $M$, or simply $\rk$ when the matroid in question is clear,
        \item $\mathscr L_M$ the lattice of flats of $M$, which we also use to denote the set of flats,
        \item $\mathfrak A(M)$ the set of atoms of $\mathscr L_M$, which are flats of rank 1,
        \item $\mathcal S(M)$ the set of spanning sets of $M$, which are subsets of $E$ that contain a basis of $M$,
	\item $\operatorname{cl}_M(S)$ the closure of a subset $S\subseteq E$, which is the smallest flat of $M$ containing $S$, and
        \item $Q(M)$ the base polytope of $M$, which is the polytope $\operatorname{Conv}(\be_B \mid B\in \mathcal B(M)) \subset \RR^E$.
\end{itemize}

\medskip
Matroids define Minkowski weights on $\Sigma_{A_n}$ in the following way.

\begin{prop}\label{prop:bergman}
Let $M$ be a loopless matroid $M$ of rank $r = d+1$.
\begin{enumerate}
\item \cite[Theorem 4.4.5]{MS15} A function $\Delta_M: \Sigma_{A_n}(d) \to \RR$ defined by
\[
\Delta_M(\operatorname{Cone}(u_{S_1}, \ldots, u_{S_d})) = \begin{cases}
1 & \textnormal{if $S_1, \ldots,  S_d$ are flats of $M$}\\
0 & \textnormal{otherwise}
\end{cases}
\]
for each chain of nonempty proper subsets $\emptyset \subsetneq S_1 \subsetneq \cdots \subsetneq S_k \subsetneq E$ is a $d$-dimensional Minkowski weight on $\Sigma_{A_n}$.
\item \cite[Proposition 5.2]{AHK18} Let $\Sigma_M$ be the smooth fan structure on the support $|\Delta_M|$ inherited from $\Sigma_{A_n}$.  That is, $\Sigma_M$ is a subfan of $\Sigma_{A_n}$ whose cones are $\operatorname{Cone}(u_{F_1}, \ldots, u_{F_k}) \subset N_\RR$, one for each chain of nonempty proper flats $\emptyset \subsetneq F_1\subsetneq \cdots\subsetneq F_k\subsetneq E$ of $M$.  Then the Bergman class $\Delta_M$ is the unique $d$-dimensional Minkowski weight on $\Sigma_M$  (up to scaling).
\end{enumerate}
\end{prop}

\begin{defn}
With notations as in \Cref{prop:bergman} above,  we call the Minkowski weight $\Delta_M$ the \textbf{Bergman class} of $M$, and we call the fan $\Sigma_M$ the \textbf{Bergman fan}\footnote{Some define the Bergman fan of $M$ as a coarser smooth fan structure on the support $|\Sigma_M|$.  A smooth fan structure on $|\Sigma_M|$ that coarsens $\Sigma_M$ corresponds to a choice of a building set on the lattice of flats $\mathscr L_M$ \cite{FS05, AK06}.  Here we will always take the smooth fan structure for $\Sigma_M$ as a subfan of $\Sigma_{A_n}$.} of $M$.
\end{defn}

We will need the following description of supports of Bergman classes in \Cref{section:nefpresentation}.

\begin{lem}\cite[Corollary 4.2.11]{MS15} \label{lem:looplessfaces}
Let $M$ be a loopless matroid, and $Q(M)$ its base polytope.  The support $|\Delta_M|$ of its Bergman class is equal to the union of cones $\sigma$ in the outer normal fan of $Q(M)$ satisfying the following condition: The corresponding face $Q(M)(\sigma)$ of $\sigma$ is a base polytope of a loopless matroid.
\end{lem}

 \medskip
The Chow ring of a matroid is defined as the Chow cohomology ring of its Bergman fan.  Explicitly, we have the following.

\begin{defn}
The \textbf{Chow ring} of a loopless matroid $M$ is the graded ring
$$A^\bullet(M):= A^\bullet(\Sigma_M) = \frac{\RR[x_F : F\in \mathscr L_M\setminus\{\emptyset, E\}]}{\langle x_Fx_{F'} \ | \ F,F' \textnormal{ incomparable}\rangle + \langle \sum_{F\supseteq a} x_F - \sum_{G\supseteq b}x_G \ | \ a,b\in \mathfrak A(M) \rangle}.$$
We call linear combinations of the variables $x_F$ \textbf{divisors} on $M$, and the elements of $A^1(M)$ \textbf{divisor classes} on $M$.  The divisor class of $\sum_{F\supseteq a} x_F$ for any atom $a\in \mathfrak A(M)$ is called the \textbf{hyperplane class} of $M$, and is denoted $\alpha(M)$ or just $\alpha$ if the matroid in question is clear.
\end{defn}

\begin{rem}
The ring $A^\bullet(M)$ was first studied in \cite{FY04} under a slightly different presentation, which for clarity is denoted $A_{FY}^\bullet(M)$ and is given as
$$A_{FY}^\bullet(M) := \frac{ \RR[z_F : F\in {\mathscr L_M}\setminus\{\emptyset\}] }{ \langle z_Fz_{F'} \ | \ F,F' \textnormal{ incomparable}\rangle + \langle \sum_{F\supseteq a}z_F | \ a\in \mathfrak A(M) \rangle }.$$
That is, we have $x_F = z_F$ for every nonempty proper flat $F\in \mathscr L_M$, and $z_E = -\alpha$.
As both presentataions $A^\bullet(M)$ and $A^\bullet_{FY}(M)$ are relevant for us, we will use the variable names $x,z$ in a consistent manner;  for example, in the summation $\sum_{F\supseteq F'} x_F$ it is implied that $F\subsetneq E$, whereas $\sum_{F\supseteq F'} z_F$ includes $z_E$ as a summand.
\end{rem}

Since $\operatorname{MW}(\Sigma_M) \simeq \RR$ by \Cref{prop:bergman}.(2), with the Bergman class as the fundamental class, the Chow ring of a matroid $M$ has the \textbf{degree map} $\int_M: A^d(\Sigma_M) \to \RR$, defined by the cap product $\xi \mapsto \xi\cap \Delta_M$.   Explicitly, it is determined by
$$\int_M x_{F_1}x_{F_2}\cdots x_{F_d} = 1 \textnormal{ for every maximal chain } F_1\subsetneq \cdots \subsetneq F_d \textnormal{ in }  {\mathscr L_M}\setminus\{\emptyset,E\}.$$

\medskip
Note that the braid fan $\Sigma_{A_n}$ is the Bergman fan of the Boolean matroid $U_{|E|,E}$, and its fundamental class $\Delta_{\Sigma_{A_n}}$ is the Bergman class of $U_{|E|,E}$.  We will thus always identify $A^\bullet(\Sigma_{A_n}) = A^\bullet(U_{|E|,E})$.

\medskip
We end this subsection with a discussion of nef and ample divisors on the braid fan, and the resulting analogous notions for Bergman fans.
The following characterization of nef divisors on $\Sigma_{A_n}$, which is a specialization of Theorem \ref{thm:nefdef}, was recognized in various works \cite{Edm70, Mur03, Pos09, AA17}; for a modern treatment and generalization to arbitrary Coxeter root systems we point to \cite{ACEP20}.

\begin{prop}\label{prop:genperm}
The following are equivalent for a divisor $D = \sum_{ \emptyset \subsetneq S \subsetneq E} c_S x_S \in A^1(\Sigma_{A_n})$.
\begin{enumerate}
\item $D$ is a nef divisor on $\Sigma_{A_n}$,
\item the function $c_{(\cdot)} : 2^{E} \to \RR$ satisfies the \textbf{submodular property}
$$c_A + c_B \geq  c_{A\cup B} + c_{A\cap B} \quad \textnormal{for every } A,B\subseteq E \textnormal{ where $c_\emptyset = c_{E} = 0$},$$
\item the normal fan of the polytope $P_D = \{m \in N^\vee_\RR \ | \ m( u_S) \leq c_S \ \forall \emptyset \subsetneq S \subsetneq E\}$ coarsens $\Sigma_{A_n}$,
\item every edge of $P_D$ is parallel to $\be_i - \be_j$ for some $i\neq j \in E$.
\end{enumerate}
\end{prop}

\begin{rem}\label{rem:genpermshift}
Often the polytope $P_D$ is constructed in an affine translate of $N^\vee_\RR$ in $\RR^E$, for which the presentation $A_{FY}^\bullet(\Sigma_{A_n})$ is useful.  Given a submodular function $c_{(\cdot)}: 2^{E} \to \ZZ$ with $c_{\emptyset} = 0$ but $c_{E}$ possibly nonzero, the \textbf{generalized permutohedron} associated to $c_{(\cdot)}$ is the polytope
$$P(c) := \{ y \in (\RR^{E})^\vee \ | \  y( \be_E) = c_{E} \textnormal{ and } y( \be_S) \leq c_{S} \ \forall \emptyset \subsetneq S \subsetneq [n]\}.$$
This polytope lives in the translate of $N_\RR^\vee$ where the points have coordinate sum $c_E$.  One translates $P(c)$ to $N^\vee_\RR$ as follows.  Fix an element $i\in E$.  We have
$$P(c) - c_{E}\be_i = \{m \in N_\RR^\vee \ | \  m( u_S) \leq c_{S} - c_{E}\alpha^{(i)}_S \ \forall S\} \subset N_\RR^\vee$$
where $\alpha^{(i)}_S = 1$ if $i\in S$ and 0 otherwise.  Since the divisor class of $\sum_{\emptyset \subsetneq S \subsetneq E}\alpha_S^{(i)}x_S$ is the hyperplane class $\alpha$ in $A^1(U_{|E|,E}) = A^1(\Sigma_{A_n})$, the nef divisor class that the polytope $P(c)$ corresponds to is
$$- c_{E}\alpha + \sum_{\emptyset \subsetneq S \subsetneq E} c_{S} x_S = \sum_{\emptyset\subsetneq S \subseteq E} c_{S} z_S.$$
\end{rem}

\medskip
The notion of nef and ample divisors on a matroid is inherited from the braid fan.  First, note that for a loopless matroid $M$, the inclusion of fans $\iota_M: \Sigma_M \hookrightarrow \Sigma_{A_n}$ induces the pullback map
$$\iota_M^*: A^\bullet(\Sigma_{A_n}) \to A^\bullet(M), \textnormal{ defined by } x_S \mapsto \begin{cases}
x_S & \textnormal{if $S\subsetneq E$ is a flat of $M$}\\
0 & \textnormal{otherwise.}
\end{cases}$$
When we wish to clarify whether a variable $x_S$ is an element of $A^\bullet(M)$ or $A^\bullet(\Sigma_{A_n})$, we write
$$x_S(M) := \iota_M^*x_S,$$
in which case $x_S$ is considered as an element of $A^\bullet(\Sigma_{A_n})$ and $x_S(M)$ of $A^\bullet(M)$.

\medskip 
The pullback map motivates the following notions regarding divisors on $M$.  We say that a divisor (class) on $M$ is \textbf{combinatorially nef} if it is a pullback of a nef divisor (class) on $\Sigma_{A_n}$.  A \textbf{combinatorially ample} divisor (class) is defined similarly.  Explicitly, a divisor $D = \sum_{F\in \mathscr L_M\setminus\{\emptyset, E\}} c_F \cdot x_F(M)$ is combinatorially nef if there exists a function $a_{(\cdot)} : 2^E \to \ZZ$ such that $a_\emptyset = a_E = 0$, $a_F = c_F$ for all flats $F\in \mathscr L_M$, and 
\[
a_A + a_B \geq  a_{A\cup B} + a_{A\cap B} \quad\textnormal{for every } A,B\subseteq E.
\]
Say $D$ is \textbf{ample} if the inequality is strict whenever $A$ and $B$ are incomparable.

Combinatorially nef (resp.\ ample) divisor classes on $M$ are closed under nonnegative linear combinations, since nef (resp.\ ample) divisor classes on complete fans are in general.  We thus let $\overline{\mathscr K}_M$ (resp.\ $\mathscr K_M$) be the cone in $A^1(M)$ of combinatorially nef (resp.\ ample) divisor classes on $M$, called the \textbf{combinatorially nef} (resp.\ \textbf{ample}) \textbf{cone} of $M$.

\begin{rem}\label{rem:combnefeff}
It follows from Remark \ref{rem:nefeff} that a combinatorially nef divisor class $[D]\in A^1(M)$ is effective; that is, it can be written as $D= \sum_F c_Fx_F$ where $c_F \geq 0$ for all $F\in \mathscr L_M\setminus\{\emptyset,E\}$.
\end{rem}

\subsection{The geometry of matroids via wonderful compactifications} \label{subsection:wondcpt}

We provide the underlying algebraic geometry of the combinatorial constructions in the previous two subsections \S\ref{subsection:ChowMW} and \S\ref{subsection:chowring}.

\medskip
Let $M$ be a loopless matroid on $E = \{0,1,\ldots, n\}$ of rank $r=d+1$ realizable over a field $\mathbbm k$, which we may assume to be algebraically closed.  A \textbf{realization} $\mathscr R(M)$ of $M$ consists of any of the following equivalent pieces of data:
\begin{itemize}
\item a list of vectors $E = \{v_0, \ldots, v_n\}$ spanning a $\mathbbm k$-vector space $V\simeq \mathbbm k^r$, or
\item a surjection $\mathbbm k^{n+1} \twoheadrightarrow V$ where $\be_i \mapsto v_i$, or
\item an injection $\PP V^* \hookrightarrow \PP_{\mathbbm k}^n$, dualizing the surjection $\mathbbm k^{n+1} \twoheadrightarrow V$.
\end{itemize}

For a realization $\mathscr R(M)$ of $M$ with $\PP V^* \hookrightarrow \PP^n$, the coordinate hyperplanes of $\PP^{n}$ intersect with $\PP V^*$ to give the associated \textbf{hyperplane arrangement} $\mathcal A_{\mathscr R(M)}$ on $\PP V^*$, which is encoded by the flats of $M$ in the following way.  For each nonempty flat $F$ of $M$, let $L_F$ be a linear subspace of $V^*$ defined by
$$L_F := \{f\in V^* \ | \ f(v_i) = 0 \ \forall v_i\in F\},$$
and let $\PP L_F$ be the linear subvariety of $\PP V^*$.  The hyperplanes of $\mathcal A_{\mathscr R(M)}$ are $\{ \PP L_a \}_{a\in \mathfrak A(M)}$ corresponding to the atoms, and more generally, a flat $F$ of rank $c$ corresponds to the $c$-codimensional linear subvariety $\PP L_F$.

\medskip
We denote by $\mathring Y_{\mathscr R(M)}$ the hyperplane arrangement complement $\PP V^* \setminus \bigcup \mathcal A_{\mathscr R(M)}$.  It is a linear subvariety of an algebraic torus in the following way: The algebraic torus $T_N = (\mathbbm k^*)^{n+1}/\mathbbm k^*$ of the lattice $N = \ZZ^{n+1}/\ZZ(1,1,\ldots, 1)$ is the complement of the union of coordinate hyperplanes in $\PP^n$, and hence $\mathring Y_{\mathscr R(M)}$ is the intersection of $\PP V^*$ with $T_N$.  The linear subvariety $\mathring Y_{\mathscr R(M)} \subset T_N$ is related to the Bergman class of $M$ by tropicalization in the following way (see \cite[\S3]{MS15} for tropicalizations of subvarieties of an algebraic torus).

\begin{thm}
Let $\mathscr R(M)$ be a realization of a loopless matroid $M$, and let $\mathring Y_{\mathscr R(M)} \subset T_N$ be the associated hyperplane arrangement complement.   Recall that the permutohedral variety $X_{A_n}$ is a toric variety with the torus $T_N$.
\begin{enumerate}
\item \cite[Theorem 4.1.11]{MS15}
The support of the tropicalization of $\mathring Y_{\mathscr R(M)}$, denoted $\trop(\mathring Y_{\mathscr R(M)})$ equals the support of the Bergman fan $\Delta_M$.
\item \cite[Proposition 6.4.17 \& Theorem 6.7.7]{MS15} It follows that the closure $Y_{\mathscr R(M)}$ of $\mathring Y_{\mathscr R(M)}$ in the permutohedral variety $X_{A_n}$ satisfies
\[
\delta_{\Sigma_{A_n}}([Y_{\mathscr R(M)}]) = \Delta_M,
\]
where $\delta_{\Sigma_{A_n}}$ is the isomorphism $A^\bullet(X_{A_n}) \simeq \operatorname{MW}_{n-\bullet}(\Sigma_{A_n})$ in \Cref{thm:FS}.
\end{enumerate}
\end{thm}

In other words, the Bergman class $\Delta_M$ corresponds to the Chow homology class of the closure $Y_{\mathscr R(M)}$ of $\mathring Y_{\mathscr R(M)}$ in the permutohedral variety $X_{{A_n}}$.  The variety $Y_{\mathscr R(M)}$ is called the \textbf{wonderful compactification} of the hyperplane arrangement complement $\mathring Y_{\mathscr R(M)}$.

\begin{rem}\label{rem:wndcpt}
The wonderful compactification $Y_{\mathscr R(M)}$ can be described in two equivalent ways \cite[\S3.2]{dCP95}.
\begin{enumerate}
\item The variety $Y_{\mathscr R(M)}$ is obtained as a series of blow-ups on $\PP V^*$ by the following process:  First blow-up the points $\{\PP L_F\}_{\operatorname{rk}(F) = \operatorname{rk}(E)-1}$, then blow-up the strict transforms of the lines $\{\PP L_F\}_{\operatorname{rk}(F) =\operatorname{rk}(E)-2}$, and continue until having blown-up strict transforms of $\{\PP L_F\}_{\operatorname{rk}(F) = 1}$.  We denote by $\pi_{\mathscr R(M)}: Y_{\mathscr R(M)} \to \PP V^*$ be the blow-down map.
\item For each nonempty flat $F$ of $M$, the projection away from the linear subvariety $\PP L_F \subset \PP V^*$ is a rational map $\PP V^* \dashrightarrow \PP(V^*/{L_F})$.  The variety $Y_{\mathscr R(M)}$ then is the (closure of) the graph of the rational map
$$\PP V^* \dashrightarrow \prod_{F\in \mathscr L_M\setminus \{\emptyset\}} \PP(V^*/{L_F}).$$
\end{enumerate}
\end{rem}

When $U_{n+1,n+1}$ is realized as the standard basis of $\mathbbm k^{n+1}$, the associated wonderful compactification is the toric variety $X_{{A_n}}$ of the braid fan.  It is obtained from $\PP^n$ by blowing up the coordinate points, then the (strict transforms of) coordinate lines, and so forth.  Let us write $\pi_{A_n}: X_{\Sigma_{A_n}} \to \PP^n$ for the blow-down map.  Then for a realization $\PP V^* \hookrightarrow \PP^n$ of a loopless matroid $M$, \Cref{rem:wndcpt}.(1) above expresses the wonderful compactification $Y_{\mathscr R(M)}$ as the strict transform of $\PP V^* \subset \PP^n$ under the sequence of blow-ups $\pi_{A_n}$.  In other words, we have a diagram
$$
\xymatrix
{
&{Y_{\mathscr R(M)}} \ar@{^{(}->}^{\quad \iota_M}[r] \ar[d]_{\pi_{\mathscr R(M)}} &X_{\Sigma_{A_n}} \ar[d]^{\pi_{A_n}}\\
&\PP V^* \ar@{^{(}->}[r] &\PP^n.
}
$$
The boundary of $Y_{\mathscr R(M)}\setminus \mathring Y_{\mathscr R(M)}$ consists of the exceptional divisors $\mathcal E_F$ obtained by blowing up (strict transforms of) $\PP{L_F}$.  These divisors have simple-normal-crossings \cite{dCP95}, and consequently the intersection theory of the boundary divisors of $Y_{\mathscr R(M)}$ is encoded in the matroid.  More precisely, the Chow ring $A^\bullet(Y_{\mathscr R(M)})$ of the variety $Y_{\mathscr R(M)}$ is isomorphic to the Chow cohomology ring $A^\bullet(M)$ of the Bergman fan of $M$ \cite[Corollary 2]{FY04}.

\begin{rem}\label{rem:matgeom}
We note the following geometric observations about the presentation
$$A^\bullet(Y_{\mathscr R(M)}) \simeq A^\bullet(M) = \frac{\RR[x_F : F\in \mathscr L_M\setminus\{\emptyset, E\}]}{\langle x_Fx_{F'} \ | \ F,F' \textnormal{ incomparable}\rangle + \langle \sum_{F\supseteq a} x_F - \sum_{G\supseteq b}x_G \ | \ a,b\in \mathfrak A(M) \rangle}.$$
\begin{enumerate}
\item The variables $x_F$ correspond to the exceptional divisors $\mathcal E_F$ obtained by blowing up (strict transforms of) $\PP{L_F}$. 
\item The quadric relations $x_Fx_{F'} = 0$ reflect that two exceptional divisors from blowing up two non-intersecting linear subspaces do not intersect.
\item The linear relations defining $A^\bullet(M)$ reflect that for any atom $a\in \mathfrak A(M)$, we have $-z_E = \alpha(M) = \sum_{F\supseteq a} x_F = \pi_{\mathscr R(M)}^* h$ where $h = c_1(\mathscr O_{\PP V^*}(1))$ is the hyperplane class of $\PP V^*$.
\item Under $A^\bullet(Y_{\mathscr R(M)}) \simeq A^\bullet(M)$ and $A^\bullet(X_{\Sigma_{A_n}}) \simeq A^\bullet(\Sigma_{A_n})$, the pullback map $\iota_M^*: A^\bullet(X_{\Sigma_{A_n}}) \to A^\bullet(Y_{\mathscr R(M)})$ along the closed embedding $\iota_M: Y_{\mathscr R(M)}\hookrightarrow X_{\Sigma_{A_n}}$ is the pullback map of Chow cohomology rings of $\Sigma_M$ and $\Sigma_{A_n}$ induced by the inclusion of fans $\Sigma_M \hookrightarrow \Sigma_{A_n}$.  In particular, the pullback map is surjective.
\item A divisor class $D\in A^1(M)$ is an combinatorially ample (nef) if and only if there exists an ample (nef) divisor class $L$ on $X_{A_n}$ such that $\iota_M^*L = D$.  Combinatorially ample (nef) divisors are ample (nef) on the variety $Y_{\mathscr R(M)}$.
\end{enumerate}
\end{rem}

\begin{rem}[Relation to \ref{obs:transport}]
The geometry and the combinatorics of the permutohedral variety $X_{A_n}$ have been widely studied in various contexts including moduli spaces \cite{LM00, BB11}, convex optimization \cite{Edm70, Mur03}, Hopf monoids \cite{DF10, AA17}, and lattice polyhedra \cite{PRW08, Pos09}.  In our case, the variety $X_{{A_n}}$ plays the role of ambient variety $X$ in \ref{obs:transport}, and the wonderful compactification $Y_{\mathscr R(M)}$ the subvariety $Y\subset X$.  In \S\ref{subsection:nefpresentation}, we introduce base-point-free divisor classes on $X_{A_n}$ that generate $A^\bullet(X_{A_n})$.
\end{rem}

\section{The simplicial presentation and its monomials}\label{section:nefpresentation}

In this section, we introduce the paper's main object of study: a new presentation of the Chow ring of a matroid which we call the simplicial presentation $\myChow^\bullet
(M)$ of $A^\bullet(M)$.


\smallskip
After a combinatorial preparation in \S\ref{subsection:matquot}, we introduce the simplicial presentation in \S\ref{subsection:nefpresentation} and show that multiplying by its generators corresponds to a matroid operation called principal truncations.  In \S\ref{subsection:nestedbasis}, we extend this correspondence to establish a combinatorial interpretation of a monomial basis of the Chow ring of a matroid.

\subsection{Matroid quotients, principal truncations, and matroid intersections}\label{subsection:matquot}

We first prepare by reviewing the relevant combinatorial notions.  We point to \cite[\S7]{Oxl11} and \cite[\S2.3]{Ham17} for further details.

\medskip
Let $M$ and $M'$ be matroids on a common ground set $E = \{0,1,\ldots, n\}$.

\begin{defn} The matroid $M'$ is a \textbf{(matroid) quotient} of $M$, written $f: M' \twoheadleftarrow M$, if any every flat of $M'$ is also a flat of $M$.  In particular, if $M$ and $M'$ are loopless, then $f: M'\twoheadleftarrow M$ if and only if $\Sigma_{M'} \subseteq \Sigma_M$.
\end{defn}

\begin{eg}
Any matroid on ground set $E$ is a quotient of the Boolean matroid $U_{|E|, E}$.  Any Bergman fan of a loopless matroid is a subfan of the braid fan.
\end{eg}

\begin{eg}[Realizable matroid quotients]
Matroid quotients model linear surjections (dually, linear injections) in the following way.  Let $M$ and $M'$ have realizations by $\kk^E \twoheadrightarrow V$ and $\kk^E \twoheadrightarrow V'$ (respectively).  If there exists a commuting diagram of linear maps
\[
\begin{tikzcd}
&{\kk^E} \arrow[r, two heads] \arrow[d, equal] &V \arrow[d, two heads]\\
&{\kk^E} \arrow[r, two heads] &V'
\end{tikzcd} \qquad \textnormal{or dually}
\begin{tikzcd}
&{\PP_\kk^n} \arrow[r, hookleftarrow] &\PP V^* \\
&{\PP_\kk^n} \arrow[r, hookleftarrow] \arrow[u, equal] &\PP {V'}^* \arrow[u, hook],
\end{tikzcd}
\]
then $M'$ is a matroid quotient of $M$.  Matroid quotients $M'\twoheadleftarrow M$ arising in this way are called {realizable matroid quotients}.  We caution that a matroid quotient $M' \twoheadleftarrow M$ with both $M'$ and $M$ realizable over the same field need not be realizable (for an example, see \cite[\S1.7.5]{BGW03}).
\end{eg}

For a matroid quotient $f: M'\twoheadleftarrow M$, the \textbf{$f$-nullity} of a subset $A \subseteq E$ is defined to be
$$n_f(A) := \rk_M(A) - \rk_{M'}(A).$$
We say that $M'$ is an \textbf{elementary (matroid) quotient} of $M$ if $n_f(E) = 1$, or equivalently if $\rk(M') = \rk(M) - 1$.  An elementary quotient of $M$ corresponds to a \textbf{modular cut} $\mathcal K$ of $M$, which is a nonempty collection of flats $\mathcal K \subset \mathscr L_M$ satisfying
\begin{enumerate}
        \renewcommand{\theenumi}{\roman{enumi}}
        \item if $F_1\in \mathcal K$ and $F_1\subset F_2$, then $F_2 \in \mathcal K$, and 
        \item if $F_1, F_2 \in \mathcal K$ and $\rk_M(F_1) + \rk_M(F_2) = \rk_M(F_1 \cup F_2) + \rk_M(F_1 \cap F_2)$, then $F_1 \cap F_2 \in \mathcal K$.
\end{enumerate}
A modular cut $\mathcal K$ of $M$ defines an elementary quotient $M' \twoheadleftarrow M$ by
\[
        \mathscr L_{M'} := \left\{ F \in \mathscr L_M : F \text{ is not covered by an element of } \mathcal K \right\} \cup \mathcal K,
\]
where $F$ is \textbf{covered} by $G$ if $F \subset G$ and $\rk(F) = \rk(G)-1$.
Conversely, given an elementary quotient $f: M' \twoheadleftarrow M$, one recovers the modular cut $\mathcal K$ of $M$ defining the elementary quotient by
\[
        \mathcal K = \{ F \in \mathscr L_{M'} : n_f(F) = 1 \}.
\]
We write $M'\overset{\mathcal K} \twoheadleftarrow M$ to denote an elementary quotient of $M$ given by a modular cut $\mathcal K$.

\begin{eg}\label{eg:modularcut}
Let $M$ have a realization $\kk^E \twoheadrightarrow V$.  For $\mathcal K$ a modular cut of $M$, let $v_{\mathcal K}$ be a nonzero vector contained in $\bigcap_{F\in \mathcal K} \operatorname{span}_\kk(F)$ and not contained in any span of a flat that is not in $\mathcal K$\footnote{For such $v_{\mathcal K}$ to exist, the field $\kk$ must be large enough, and the elementary matroid quotient defined by the modular cut $\mathcal K$ must be realizable.}.  Dually, with the notation as in \S\ref{subsection:wondcpt}, we have a general hyperplane $H_{\mathcal K} = \{f\in V^* \mid f(v_{\mathcal K}) = 0\}$ in $V^*$ containing $\bigcup_{F\in \mathcal K} L_F$.  Let us consider the commuting diagram
\[
\begin{tikzcd}
&{\kk^E} \arrow[r, two heads] \arrow[d, equal] &V \arrow[d, two heads]\\
&{\kk^E} \arrow[r, two heads] &V/\operatorname{span}_\kk(v_{\mathcal K})
\end{tikzcd} \quad \textnormal{or dually,}
\begin{tikzcd}
&{\PP_\kk^E} \arrow[r, hookleftarrow] &\PP V^* \\
&{\PP_\kk^E} \arrow[r, hookleftarrow] \arrow[u, equal] &\PP {H_{\mathcal K}}.\arrow[u, hook]
\end{tikzcd}
\]
The map $\kk^E \twoheadrightarrow V/\operatorname{span}_\kk(v_{\mathcal K})$ is a realization of the matroid $M'$ of the elementary quotient $M' \overset{\mathcal K}\twoheadleftarrow M$ defined by $\mathcal K$.  Dually, with the notation as in \S\ref{subsection:wondcpt}, the associated hyperplane arrangement $\mathcal A_{\mathscr R(M')}$ is the intersection of $\PP H_{\mathcal K}$ with the coordinate hyperplanes in $\PP^E_\kk$.  Equivalently, the hyperplane arrangement $\mathcal A_{\mathscr R(M')}$ is the intersection of $\PP H_{\mathcal K}$ with the hyperplanes in the hyperplane arrangement $\mathcal A_{\mathscr R(M)}$ under the inclusion $\PP H_{\mathcal K} \subset \PP V^*$.
\end{eg}

Of particular interest in our case is when $\mathcal K$ is the interval $[F,E] \subset \mathscr L_M$, since an interval in $\mathscr L_M$ is always a modular cut.  We call the resulting elementary quotient, denoted $T_F(M)$, the \textbf{principal truncation} of $M$ associated to the flat $F$.  An explicit description of principal truncations follows.

\begin{prop}\cite[Exercise 7.2.4.]{Oxl11}\label{prop:quotbases}
The principal truncation $T_F(M)$ of a matroid $M$ associated to a flat $F \in \mathscr L_M$ of rank at least 1 has bases
$$\mathcal B(T_F(M)) = \{ B \setminus f \textnormal{ such that } B \in \mathcal B(M) \textnormal{ and } f \in B \cap F \neq \emptyset\},$$
and the flats of $T_F(M)$ partition into two sets $\mathscr L_{T_F(M)} = \mathcal K \sqcup \mathcal L$ according to their $f$-nullities by
$$
\begin{array}{l}
\mathcal K = \{G\in \mathscr L_{T_F(M)} \ | \ n_f(G) = 1\} = \{ G\in \mathscr L_M \ | \ F\subseteq G\},\\
\mathcal L = \{G\in \mathscr L_{T_F(M)} \ | \ n_f(G) = 0\} = \{G\in \mathscr L_M \ | \ \textnormal{$G$ not covered by an element in $[F,E]$}\}.
\end{array}
$$
\end{prop}

\begin{rem}\label{eg:geomquot}
In \Cref{eg:modularcut}, if $\mathcal K = [F,E]$ for some flat $F$, then we can set $v_{\mathcal K} = v_F$, a general vector in $\operatorname{span}_\kk(F)$, and dually, we can set $H_{\mathcal K} = H_F$, a general hyperplane in $V^*$ containing $L_F$.
\end{rem}

We end our combinatorial preparation by connecting matroid quotients to Minkowski weights on braid fans via the notion of matroid intersections.  Recall that a spanning set of a matroid is a subset that contains a basis of the matroid.

\begin{defn}
The \textbf{matroid intersection}\footnote{Matroid intersection, as defined here and \cite{Wel76, Bry86}, is related to but different from the same terminology used in \cite{Edm70, Law75}, where one considers pairwise intersections of independent sets of two matroids.} of two matroids $M$ and $N$ on a common ground set $E$ is a new matroid $M\wedge N$ on $E$ whose spanning sets $\mathcal S(M\wedge N)$ are $\{S\cap S' \ | \ S\in \mathcal S(M),\ S'\in \mathcal S(N)\}$.
\end{defn}

The matroid $M\wedge N$ is a matroid quotient of both $M$ and $N$.  Matroid intersection behaves well in relation to Minkowski weights in the following way.  Recall that the isomorphism $A^\bullet(\Sigma_{A_n}) \simeq \operatorname{MW}_{n-\bullet}(\Sigma_{A_n})$ of \Cref{thm:FS} makes $\operatorname{MW}_{n-\bullet}(\Sigma_{A_n})$ into a graded ring, with multiplication called the stable intersection $\cap_{st}$.  The following proposition states that stable intersections of Bergman classes are Bergman classes of matroid intersections.

\begin{prop} \cite[Proposition 4.4]{Spe08}, \cite[Remark 2.31]{Ham17}\label{prop:matprod}
Let $M$ and $N$ be two matroids on a common ground set $E$, and let $\Delta_M$ and $\Delta_N$ be their Bergman classes, which are Minkowski weights on $\Sigma_{A_n}$.  Then we have 
$$\Delta_M \cap_{st} \Delta_N = 
\left\{\begin{array}{ll}
\Delta_{M\wedge N} & \textnormal{if $M\wedge N$ is loopless}\\
0 & \textnormal{otherwise.}\end{array}\right.$$
\end{prop}

\subsection{The variables of the simplicial presentation}\label{subsection:nefpresentation} We now define a new presentation $\myChow^\bullet(M)$ of the Chow ring of a matroid $M$, and discuss its first properties.  The key result here is that the variables of $\myChow^\bullet(M)$ correspond to principal truncations of $M$.

\medskip
We prepare by noting a distinguished set of nef divisor classes on $\Sigma_{A_n}$ and their polytopes considered in \cite{Pos09}.  For a nonempty subset $S$ of $E$, let
$$\nabla_S := \operatorname{Conv}(-\be_i \ | \ i \in S) \subset \RR^E$$
denote the negative standard simplex of $S$.  As the edges of $\nabla_S$ are parallel translates of $\be_i - \be_j$ for $i\neq j \in S$, Proposition \ref{prop:genperm} (in the form of \Cref{rem:genpermshift}) implies that $\nabla_S$ is a polytope with the corresponding nef divisor class
$$h_S := -\sum_{S\subseteq T} z_T \in A_{FY}^1(\Sigma_{A_n}).$$
These divisors were considered in \cite{Pos09}\footnote{In \cite{Pos09} the author uses $y_S$ to denote the nef divisor of the standard simplex of $S$ instead of the negative standard simplex.  This difference is due to using inner normal fans instead of outer normal fans of polytopes.} and implicitly in \cite{Ham17}.  We now consider the presentation of $A^\bullet(M)$ given by pullbacks of these nef divisor classes corresponding to (negative) standard simplices.  For $M$ a loopless matroid on $E$, and $\emptyset \neq S\subseteq E$, denote $h_S(M) := \iota_M^*h_S$.  If $F = \operatorname{cl}_M(S)$ is the smallest flat containing $S$, note that we have
\begin{equation}\label{eqn:S=F}
h_S(M) := \iota_M^* h_S = - \sum_{S \subseteq T} z_T(M) = - \sum_{F\subseteq G \in \mathscr L_M} z_G(M) = \iota_M^* h_F,
\end{equation}
as $z_T(M) = \iota_M^*z_T = 0$ for all $T\subseteq E$ not a flat of $M$.  By construction, the elements $h_F(M) \in A^1(M)$ are (combinatorially) nef divisor classes on $M$.  We will simply write $h_F$ for $h_F(M)$ when there is no confusion.

\begin{defn}\label{dfn:simplicial}
For $M$ a loopless matroid on $E$, the \textbf{simplicial presentation} $\myChow^\bullet(M)$ of the Chow ring of $M$ is the presentation of $A^\bullet(M)$ whose generators are $\{h_F\}_{F\in \mathscr L_M\setminus\{\emptyset\}}$ where
$$h_F := -\sum_{F \subseteq G} z_G \in A_{FY}^\bullet(M).$$
We call $\{h_F\}_{F\in \mathscr L_M\setminus\{\emptyset\}}$ the \textbf{simplicial generators}.
\end{defn}

The variable $h$ here stands for ``hyperplane''; for the geometric origin of the simplicial presentation see Remarks \ref{rem:geomhyper} and \ref{rem:simplicialpullback}.  The linear change of variables from $\{z_F\}_{F\in \mathscr L_M\setminus\{\emptyset\}}$ to $\{h_F\}_{F\in \mathscr L_M\setminus\{\emptyset\}}$ is evidently invertible, given by an upper triangular matrix.  Explicitly, by M\"obius inversion we have
$$-z_F = \sum_{F \subseteq G} \mu(F, G) h_G$$
where $\mu$ is the M\"obius function on the lattice $\mathcal L_M$. Thus, the explicit presentation of $\myChow^\bullet(M)$ is
\[
        \myChow^\bullet(M) :=  \RR[h_F : F \in \mathscr L_M \setminus \{\emptyset\} ] \,/\, (I + J)
\]
where
\[
        I = \left\langle h_a : a \in \mathfrak A(M) \right\rangle \textnormal{ and } J = \left\langle \left( \sum_{F \subseteq G} \mu(F, G) h_G \right) \left( \sum_{F' \subseteq G'} \mu(F', G') h_{G'}\right) : F, F' \text{ incomparable} \right\rangle.
\]
Denote by $\mathscr L_M^{\geq 2}$ the set of flats of $M$ of rank at least 2.  Noting that $h_a = 0 \in \myChow^\bullet(M)$ for any atom $a\in \mathfrak A(M)$, we define $\{h_F \ | \ F\in \mathscr L_M^{\geq 2}\}$ to be the \textbf{nontrivial simplicial generators} of the Chow ring of $M$.  They form a basis of $\myChow^1(M)$.

\begin{rem}\label{rem:EHKRS}
When the matroid $M$ is the cyclic matroid of the complete graph $K_{n-1}$ on $n-1$ vertices, the Chow ring of $M$ is the cohomology ring of the Deligne-Knudson-Mumford space $\overline{\mathcal M}_{0,n}$ of rational curves with $n$ marked points \cite[\S4.3]{dCP95}, \cite[Theorem 6.4.12]{MS15}.  In this case, by using the minimal building set instead of the maximal building set, the simplicial presentation recovers the Etingof-Henriques-Kamnitzer-Rains-Singh presentation of the cohomology ring of $\overline{\mathcal M}_{0,n}$ \cite{EHKR10, Sin04}.  In this presentation, the author of \cite{Dot19} showed that the cohomology ring of $\overline{\mathcal M}_{0,n}$ is Koszul because it has a quadratic Gr\"obner basis.  In the classical presentation, the Chow ring of any matroid with rank $>3$ has no quadratic Gr\"obner basis.
\end{rem}

The following theorem, which relates the variables of the simplicial presentation to principal truncations, is the key property of the simplicial presentation that we use throughout this paper.  Let us first fix a notation:

\proof[Notation] For a nonempty subset $S\subseteq E$, we denote by $H_S$ the matroid with bases
$$\mathcal B(H_S) := \{ E \setminus i : i \in S\},$$
or equivalently, $H_S = U_{|E \setminus S|, E\setminus S} \oplus U_{|S|-1, S}$.

\begin{thm}\label{thm:hyper}
Let $M$ be a loopless matroid on $E$, and $S$ a nonempty subset of $E$.  
Write $F$ for the smallest flat of $M$ containing $S$.  Then $H_S \wedge M = T_F(M)$, and the nef divisor class $h_S \in \myChow^1(\Sigma_{A_n})$ satisfies 
\[
h_S \cap \Delta_M =\begin{cases} \Delta_{T_F(M)} &\textnormal{if $\operatorname{rk}_M(S) > 1$}\\ 0 &\textnormal{otherwise}.
\end{cases}
\]
\end{thm}

The theorem will follow mostly from the following lemma.

\begin{lem} \label{lem:multquot}
Let $M$ and $H_S$ be as in the statement of Theorem \ref{thm:hyper}, and let  $F$ be the smallest flat of $M$ containing $S$.   Then we have
\begin{equation}\label{eq:intersectionTruncation}
H_S \wedge M = T_F(M),
\end{equation}
and consequently,
\begin{equation}\label{eq:minkowskiCap}
 \Delta_{H_F}\cap_{st}\Delta_M  = 
\left\{\begin{array}{ll}
\Delta_{T_F(M)} & \textnormal{if $\operatorname{rk}_M(F) > 1$}\\
0 & \textnormal{otherwise.}\end{array}\right.
\end{equation}
\end{lem}

\begin{proof}
We first show that \eqref{eq:minkowskiCap} follows from \eqref{eq:intersectionTruncation}.
A matroid is loopless if and only if the empty set $\emptyset$ is a flat. By \Cref{prop:quotbases}, the matroid $T_F(M)$ is thus loopless if and only if $\emptyset$ is not covered by an element in the interval $[F,E]$. As claimed, this happens if and only if $\operatorname{rk}_M(F) > 1$.

We now prove \eqref{eq:intersectionTruncation}.
By definition of $H_S \wedge M$, the minimal elements in the set of spanning sets $\mathcal S(H_S \wedge M)$ are $B\setminus i$ where $B\in \mathcal B(M)$ and $i\in B\cap S \neq \emptyset$.  Since minimal spanning sets are bases, we have
\[
\mathcal B(H_S \wedge M) = \{B\setminus i \textnormal{ such that } B\in \mathcal B(M),\ i \in B\cap S \neq \emptyset\}.
\]  When $S = F$, this is the description of the bases of $T_F(M)$ in Proposition \ref{prop:quotbases}, so it remains to show $H_S \wedge M = H_F \wedge M$.  Evidently, we have $\mathcal B(H_S \wedge M) \subseteq \mathcal B(H_F \wedge M)$ since $S\subseteq F$.  For the other inclusion, suppose we have a basis $B\setminus f$ of $H_F \wedge M$ where $B\in \mathcal B(M)$ and $f\in B\cap F \neq \emptyset$.  
We show that there is an element $s\in S$ such that $B' = (B\setminus f)\cup s$ is also a basis of $M$, which implies that $B\setminus f = B' \setminus s$ is a basis of $H_S\wedge M$.
First, because $F$ is the closure of $S$, we have $\operatorname{rk}_M( (B\setminus f) \cup S) = \operatorname{rk}_M( (B\setminus f) \cup F)$.  We also have $ \operatorname{rk}_M( (B\setminus f) \cup F) = \operatorname{rk}_M(B \cup F) = \operatorname{rk}_M(E)$ since $B$ is a basis.  Since $B\setminus f$ is independent in $M$, we thus conclude from $\operatorname{rk}_M( (B\setminus f) \cup S) = \operatorname{rk}_M(E)$ that $B\setminus f$ can be extended to a basis by an element in $S$.  That is, there exists an element $s\in S$ such that $(B \setminus f) \cup s$ is a basis of $M$, as desired.
\end{proof}

\begin{proof}[Proof of \Cref{thm:hyper}]
Let $\delta_{\Sigma_{A_n}}: A^\bullet(\Sigma_{A_n}) \overset\sim\to \operatorname{MW}_{n-\bullet}(\Sigma_{A_n})$ be the isomorphism map in \Cref{thm:FS}.  We claim that $\delta_{\Sigma_{A_n}}(h_S) = \Delta_{H_S}$, which is proved in \Cref{lem:hyper} below.  Our desired statement then follows immediately from \Cref{lem:multquot}, since $h_S \cap \Delta_M = (\delta_{\Sigma_{A_n}}(h_S)) \cap_{st} \Delta_M$ by the definition of stable intersection $\cap_{st}$.
\end{proof}

\begin{lem}\label{lem:hyper}
Let $\delta_{\Sigma_{A_n}}: A^\bullet(\Sigma_{A_n}) \overset\sim\to \operatorname{MW}_{n-\bullet}(\Sigma_{A_n})$ be the isomorphism map in \Cref{thm:FS}.  Then we have
\[
\delta_{\Sigma_{A_n}}(h_S) = \Delta_{H_S}.
\]
\end{lem}

\begin{proof}
We claim that the support $|\Delta_{H_S}|$ of $\Delta_{H_S}$ is equal to the support of the $(n-1)$-skeleton of the outer normal fan of negative standard simplex $\nabla_S$.  If this is the case, then \Cref{prop:trophyper} implies $\delta_{\Sigma_{A_n}}(h_S) = \Delta_{H_S}$ because all the edges of the negative standard simplex $\nabla_S$ have lattice length 1.

Now, for the claim, note first that the translate $\nabla_S + \be_E$ of $\nabla_S$ is
$\operatorname{Conv}(\be_{E\setminus i} \ |\ i \in S) \subset \RR^E,$ 
which is equal to the base polytope $Q(H_S)$ of $H_S$.  Since every face of $Q(H_S)$, except for the vertices, are base polytopes of loopless matroids, by \Cref{lem:looplessfaces} the support of $|\Delta_{H_S}|$ equals the support of the $(n-1)$-dimensional skeleton of the outer normal fan of $\nabla_S$.
\end{proof}

\Cref{thm:hyper} encodes the combinatorics of the following geometric motivation for the simplicial presentation.

\begin{rem}\label{rem:geomhyper}
We recall the following standard fact in algebraic geometry.  Let $L\subset V^*$ be an inclusion of vector spaces, so that $\PP L$ is a linear subvariety of $\PP V^*$.  Let $X = \operatorname{Bl}_{\PP L}\PP V^*$ be the blow-up of $\PP V^*$ along $\PP L$, which is also the graph of the rational map $\PP V^* \dashrightarrow \PP(V^*/L)$, and let $\pi: X \to \PP V^*$ be the blow-down map.   Let $h = c_1(\mathscr O_{\PP V^*}(1))$ denote the hyperplane class, and $\mathcal E$ denote the exceptional divisor of the blow-up.  Then the map $\operatorname{Bl}_{\PP L} \PP V^* \to \PP(V^* /L)$ corresponds to the linear series $| \mathscr O_X(\pi^* h - \mathcal E)|$, so the sections of the line bundle $\mathscr O_X(\pi^* h - \mathcal E)$ correspond to the hyperplanes in $\PP V^*$ that contain $\PP L$.

Now, suppose $M$ has a realization $\mathscr R(M)$ over an algebraically closed field $\mathbbm k$ as $\PP V^* \hookrightarrow \PP^n$, and let notations be as in \S\ref{subsection:wondcpt}.  The geometry of $A^\bullet(Y_{\mathscr R(M)}) \simeq A^\bullet(M)$ in \Cref{rem:matgeom} implies
\[
h_F = \sum_{G\supseteq F} - z_G = -z_E - \sum_{G\supseteq F} x_G = \pi_{\mathscr R(M)}^* h - \sum_{G\supseteq F} \mathcal E_G,
\]
and hence $h_F$ represents the divisor class of the strict transform of a general hyperplane in $\PP V^*$ containing the linear subvariety $\PP L_F$.  Thus multiplying by $h_F$ corresponds to intersecting by a general hyperplane in $\PP V^*$ containing $\PP L_F$, which corresponds to the principal truncation $T_F(M)$ by \Cref{eg:modularcut} (in the form of \Cref{eg:geomquot}).  More precisely, we have $h_F \cdot [Y_{\mathscr R(M)}] = [Y_{\mathscr R(T_F(M))}] \in A^\bullet(X_{\Sigma_{A_n}})$.  \Cref{thm:hyper} is the combinatorial mirror of this geometric observation.
\end{rem}

\begin{rem}[Relation to \ref{obs:transport}]\label{rem:simplicialpullback} Suppose $M$ has a realization $\mathscr R(M)$ by $\PP V^* \hookrightarrow \PP^n$.  By the second description in \Cref{rem:wndcpt}.(2), the wonderful compactification $Y_{\mathscr R(M)}$ is embedded in the product of projective spaces $\prod_{F\in \mathscr L_M \setminus \{\emptyset\}} \PP(V^*/L_F)$.  We described $h_F$ as a divisor class represented by the strict transform of a general hyperplane in $\PP V^*$ containing $\PP L_F$ in the previous \Cref{rem:geomhyper}.  Alternatively, the variable $h_F$ thus represents the base-point-free divisor obtained as the hyperplane class pullback of the map $Y_{\mathscr R(M)} \to \PP(V^* / L_F)$. In other words, the divisor classes $h_S \in \myChow^\bullet(M)$ play the role of ``base-point-free divisor classes,'' and we have interpreted the hyperplane class pullbacks of the maps they define as principal truncations.
\end{rem}

\begin{rem}
In the classical presentation $A^\bullet(\Sigma_{A_n})$, the cap product $x_S \cap \Delta_M$ is almost never a Bergman class of a matroid---it is a Minkowski weight which may have negative weights on some cones.  This reflects the geometry that the divisor $x_S$ is effective but usually not nef.
\end{rem}

\subsection{A monomial basis of the simplicial presentation and relative nested quotients}\label{subsection:nestedbasis}

We introduce the notion of relative nested quotients, which are relative generalizations of (loopless) Schubert matroids in matroid theory, and we show that they are in 
in bijection with elements of a monomial basis of $\myChow^\bullet(M)$.

\medskip
We start by producing a monomial basis of $\myChow^\bullet(M)$ via the Gr\"obner basis computation in \cite{FY04}.  Pick a total order $>$ on elements of $\mathscr  L_M$ such that $F>G$ if $\operatorname{rk}_M(F) \leq \operatorname{rk}_M(G)$, and take the induced lex monomial order on $A^\bullet_{FY}(M)$.  A Gr\"obner basis for $A^\bullet_{FY}(M)$ was given as follows.

\begin{thm}\cite[Theorem 1]{FY04}\label{thm:fy04gb}
The following form a Gr\"obner basis for the ideal of $A^\bullet_{FY}(M)$:
$$
\left\{\begin{array}{lr}
z_Fz_G &\textnormal{ $F$ and $G$ are incomparable nonempty flats} \\
z_F\left( \sum_{H \geq G} z_H \right)^{\operatorname{rk} G - \operatorname{rk} F}  & \textnormal{ $F\subsetneq G$ nonempty flats} \\
\left(\sum_{H \geq G} z_H\right)^{\operatorname{rk} G}  & \textnormal{$G$ a nonempty flat}
\end{array}\right. .
$$
\end{thm}
In \cite{FY04}, the authors associate a ring to an atomic lattice with a chosen ``building set,'' and \cite[Theorem 1]{FY04} provides a Gr\"obner basis for the defining ideal of these more general rings.
In our setting, the atomic lattice is the lattice of flats of a matroid, and the building set is the collection of all nonempty flats.
We now note that the Gr\"obner basis in \cite{FY04} carries over to the simplicial presentation as follows.  Again, pick a total ordering $>$ of $\mathscr L_M$ such that if $\rk_M(F) \leq \rk_M(G)$, then $F > G$.

\begin{prop}\label{prop:gbH}
        The following is a Gr\"obner basis for the defining ideal of $\myChow^\bullet(M)$ with respect to the lex monomial ordering induced by $>$:
        \[
\left\{\begin{array}{lc}
(\sum_{F\subseteq G}\mu(F,G)h_G)(\sum_{F'\subseteq G'}\mu(F',G')h_{G'}) & F,F' \textnormal{ incomparable} \\[3mm]
(\sum_{F\subseteq G}\mu(F,G)h_G) \cdot  h_{F'}^{\operatorname{rk} F' - \operatorname{rk} F}  & \textnormal{ $F\subsetneq F'$} \\[3mm]
h_F^{\operatorname{rk} F}  & F\in \mathcal L_M\setminus\{\emptyset\}
\end{array}\right. .
        \]
\end{prop}
\begin{proof}
\newcommand{\ph}{\varphi}
        Let $S_{FY} = \RR[z_F : F \in \mathscr L_M \setminus \{\emptyset \} ]$ and $S_\nabla = \RR[h_F : F \in \mathscr L_M \setminus \{\emptyset \} ]$, and define $\ph: S_{FY} \to S_\nabla$ to be the substitution $z_F \mapsto - \sum_{F \subseteq G} \mu(F, G) h_G$. 
        
        Observe that $\ph$ is lower triangular with $-1$'s on the diagonal when the variables $z_F$ and $h_F$ are written in descending order with respect to $>$.
        Hence, if $f \in S$ with initial monomial $z_{F_1}^{e_1} \cdots z_{F_k}^{e_k}$, then the initial monomial of $\ph(f)$ is $h_{F_1}^{e_1} \cdots h_{F_k}^{e_k}$.
        The proposition now follows from the fact that the elements of the Gr\"obner basis above are the images under $\ph$ of the elements of the Gr\"obner basis given in Theorem \ref{thm:fy04gb}.
\end{proof}
As a result, we obtain a monomial basis of $\myChow^\bullet(M)$.
\begin{cor}\label{cor:monombasis}
For $c \in \ZZ_{\geq 0}$, a monomial $\RR$-basis for the degree $c$ part $\myChow^c(M)$ of the Chow ring $\myChow^\bullet(M)$ of a matroid $M$ is
$$\{h_{F_1}^{a_1} \cdots h_{F_k}^{a_k} \ | \ \sum a_i = c,\ \emptyset = F_0 \subsetneq F_1 \subsetneq \cdots \subsetneq F_k,\ 1\leq a_i < \operatorname{rk}_M(F_i) - \operatorname{rk}_M(F_{i-1})\}.$$
We call this basis of $\myChow^\bullet(M)$ the \textbf{nested basis} of the Chow ring of $M$.
\end{cor}
\begin{proof}
If $B$ is a Gr\"obner basis for an ideal $I$ of a polynomial ring $R$ over a field $k$, then the monomials not divisible by the initial term of any element of $B$ are a $k$-basis for $R/I$ \cite[Theorem 39.6]{Pee11}.
Applying this fact to the Gr\"obner basis of \Cref{prop:gbH} yields the result.
\end{proof}

While the nested basis of $\myChow^\bullet(M)$ looks identical to the one given for $A_{FY}^\bullet(M)$ in \cite[Corollary 1]{FY04}, we show here that with the simplicial presentation the monomials in the basis now allow for a combinatorial interpretation as a distinguished set of matroid quotients of $M$.

\medskip
Let $f: M' \twoheadleftarrow M$ be a matroid quotient on a ground set $E$.

\begin{defn}\label{dfn:relnestedquot}
        An \textbf{$f$-cyclic flat} of $f$ is a flat $F \in \mathscr L_{M'}$ such that $F$ is minimal (with respect to inclusion) among the flats $F' \in \mathscr L_{M'}$ such that $n_f(F') = n_f(F)$.  A matroid $M'$ is a \textbf{relative nested quotient} of $M$ if the $f$-cyclic flats of $M'$ form a chain.
\end{defn}

Relative nested quotients are relative generalizations of (loopless) Schubert matroids:

\begin{eg}
        If $M = U_{|E|, E}$ then any matroid $M'$ is a quotient $f: M' \twoheadleftarrow M$.  In this case, we have $n_f(A) = |A| - \operatorname{rk}_{M'}(A)$ for a subset $A\subseteq E$.  So, if further a subset $B\subsetneq A$ satisfies $n_f(B) = n_f(A)$, then $|A| - |B| = \operatorname{rk}_{M'}(A) - \operatorname{rk}_{M'}(B)$, which occurs if and only if $A\setminus B$ is a set of coloops in the restriction $M'|_A$.  In other words, the $f$-cyclic flats of $M'$ are precisely the \textbf{cyclic flats} of $M'$, which are flats $F$ of $M'$ such that $M'|_F$ has no coloops.  Moreover, the relative nested quotients of $U_{|E|, E}$ are called \textbf{nested matroids}, which in the literature also go by (loopless) \textbf{Schubert matroids} because they are realized by general points in the appropriate Schubert subvarieties of Grassmannians.  See \cite[\S2.2]{Ham17} for more on cyclic flats and nested matroids.
\end{eg}

The data of cyclic flats of a matroid and their ranks determine the matroid \cite[Proposition 2.1]{Bry75}.  We generalize the statement to $f$-cyclic flats of a matroid quotient.  We first need the following fact about obtaining any matroid quotient as a sequence of elementary quotients.

\begin{lem}\label{lem:higgs} \ 
\begin{enumerate}
\item \cite{Hig68}, \cite[Exercise 7.20]{Bry86}
Any matroid quotient $f: M'\twoheadleftarrow M$ can be obtained as a sequence of elementary quotients in a canonical way called the \textbf{Higgs factorization} of $f$.  The Higgs factorization of a quotient $f: M' \twoheadleftarrow M$ with $n_f(E) = c$ is a sequence of elementary quotients
$$M' = M_0 \overset{\mathcal K_1}\twoheadleftarrow M_1 \overset{\mathcal K_2}\twoheadleftarrow \cdots \overset{\mathcal K_c}\twoheadleftarrow M_c = M$$
where the bases of $M_i$ for $i = 1, \ldots, c$ are defined as
\[\mathcal B(M_i) = \{A\subseteq E \ | \ \textnormal{$A$ spanning in $M'$, independent in $M$, and } |A| = \operatorname{rk}(M') + i\}.
\]
\item \cite[Theorem 3.4]{KK78}
The modular cuts $\mathcal K_i$ of the Higgs factorization are
\[\mathcal K_i = \{G\in \mathscr L_{M_i} \ | \ n_f(G) \geq i\}.
\]
\end{enumerate}
\end{lem}

\begin{prop}\label{prop:cyclic}
The data of the $f$-cyclic flats, their $f$-nullities, and the matroid $M$ determine the quotient $f: M' \twoheadleftarrow M$.  More precisely, writing $n_f(E) = c$, the data recovers the Higgs factorization $
M' = M_0 \overset{\mathcal K_1}\twoheadleftarrow M_1 \overset{\mathcal K_2}\twoheadleftarrow \cdots \overset{\mathcal K_c}\twoheadleftarrow M_c = M 
$ of $f$ by specifying the modular cuts $\mathcal K_i$ to be
\[
\mathcal K_i = \{G\in \mathscr L_{M_i} \mid G \supseteq F \textnormal{ for some } F\in \operatorname{cyc}(f) \textnormal{ with } n_f(F) \geq i\}
\]
for each $i = 1, \ldots, c$.
\end{prop}
\begin{proof}
For each $i = 1, \ldots, c$, the modular cut $\mathcal K_i$ is $\{G\in \mathscr L_{M_i} \ | \ n_f(G) \geq i\}$ by \Cref{lem:higgs}.(2).   This can equivalently be written as $\{G\in \mathscr L_{M_i} \ | \ G\supseteq F \ \textnormal{ for some } F\in \operatorname{cyc}(f) \textnormal{ with } n_f(F)\geq i\}$ by the definition of $f$-cyclic flats.
\end{proof}

We now show that the nested basis of $\myChow^\bullet(M)$ given in Corollary \ref{cor:monombasis} is in bijection with the set of relative nested quotients of $M$.

\begin{thm}\label{thm:monomnested}
Let $M$ be a loopless matroid of rank $r=d+1$.  For each $0\leq c \leq d$, the cap product map
$$\myChow^c(M) \to \operatorname{MW}_{d-c}(\Sigma_M), \quad \xi\mapsto \xi \cap \Delta_M$$
induces a bijection between the monomial basis for $\myChow^c(M)$ given in Corollary \ref{cor:monombasis} and the set of Bergman classes $\Delta_{M'}$ of loopless relative nested quotients $M'\twoheadleftarrow M$ with $\operatorname{rk}(M') = \operatorname{rk}(M) - c$.
\end{thm}

Let us prepare with a lemma.  For a flat $F$ of a matroid $M$, and an integer $a \geq 0$, denote by ${T_F}^a(M)$ the matroid obtained from $M$ by performing  $a$ many times the principal truncation associated to $F$.  Such iteration of $T_F$ is well-defined because principal truncation associated to $F$ preserves $F$ as a flat.

\begin{lem}\label{lem:monomnested}
Let $f: M' \twoheadleftarrow M$ be a matroid quotient formed by two loopless matroids, and let $F$ be a nonempty flat of $M'$.  Then, the following hold:
\begin{enumerate}[label = (\arabic*)]
\item \label{monomnested:flat} Suppose $G$ is another flat of $M'$ such that $G\subsetneq F$, and $a < \operatorname{rk}_{M'}(F) - \operatorname{rk}_{M'}(G)$.  Then, the subset $G$ is a flat of ${T_F}^a(M')$.  In particular, setting $G = \emptyset$, we have that ${T_F}^a(M')$ is loopless if $a < \operatorname{rk}_{M'}(F)$.
\item \label{monomnested:cyc} Suppose $F\subseteq S$ for every nonempty $f$-cyclic flat $S$ of $M'$.  Denote by $g$ the matroid quotient $g: {T_F}(M') \twoheadleftarrow M$.  Then, the set $\operatorname{cyc}(g)$ of $g$-cyclic flats equals $\{F\} \cup \operatorname{cyc}(f)$.
\end{enumerate}
\end{lem}

\begin{proof}
For statement (1), we first recall the description of the flats in a principal truncation (\Cref{prop:quotbases}):  A flat $G'$ of $M'$ satisfying $G'\subsetneq F$ remains a flat in $T_F(M')$ if $F$ does not cover $G'$ in $M'$.    Moreover, the rank of $F$ decreases by 1 each time one performs the principal truncation $T_F$.  Hence, the condition $a < \operatorname{rk}_{M'}(F) - \operatorname{rk}_{M'}(G)$ ensures that $G$ is not covered by $F$ in the matroid ${T_F}^{a-1}(M')$, and so $G$ is a flat of ${T_F}^a(M')$.  For statement (2), again by \Cref{prop:quotbases}, we have that the set $[F,E] = \{F' \mid F'\in \mathscr L_{M'} \text{ and } F'\supseteq F\}$ is exactly the set of flats of ${T_F(M')}$ satisfying $\operatorname{rk}_{M'}(F) - \operatorname{rk}_{T_F(M')}(F') = 1$, and the other flats of $T_F(M')$ do not change in rank when considered as flats of $M'$.  Thus, since $F$ is contained in every nonempty $f$-cyclic flat of $M'$, the subset $F$ is a $g$-cyclic flat with $n_g(F) = n_f(F)+1$, and if $S$ is a nonempty $f$-cyclic flat with $n_f(S) = k$, then $S$ is a $g$-cyclic flat with $n_g(S) = k+1$.
\end{proof}

\begin{proof}[Proof of \Cref{thm:monomnested}]
Let $h_{F_1}^{a_1} \cdots h_{F_k}^{a_k}$ be an element of the monomial basis given in Corollary \ref{cor:monombasis}.
By repeated application of Theorem \ref{thm:hyper} combined with \Cref{lem:monomnested}.\ref{monomnested:flat}, we have
\[
        h_{F_1}^{a_1} \cdots h_{F_k}^{a_k} \cap \Delta_M = \Delta_{M'}
\]
where
$M'$ is the loopless matroid obtained from $M$ by a sequence of principal truncations, first by $F_k$ repeated $a_k$ times, then by $F_{k-1}$ repeated $a_{k-1}$ times, and so forth.
Moreover, \Cref{lem:monomnested}.\ref{monomnested:cyc} implies that $f: M' \twoheadleftarrow M$ is a matroid quotient with $\operatorname{cyc}(f) =\{ \emptyset, F_1, \ldots, F_k\}$ and $n_f(F_j) = \sum_{i = 1}^j a_i$.   We have thus shown that an element of the nested basis defines a loopless relative nested quotient by the cap product.

Conversely, let $f: M' \twoheadleftarrow M$ be a loopless relative nested quotient with $\operatorname{cyc}(f) =  \{\emptyset\subsetneq F_1 \subsetneq \cdots \subsetneq F_k\}$.  Define integers $a_1, \ldots, a_k$ by $n_f(F_j) = \sum_{i = 1}^j a_i$ for all $j = 1, \ldots, k$.  \Cref{prop:cyclic} implies that the $f$-cyclic flats and their $f$-nullities of a relative nested quotient $f: M'\twoheadleftarrow M$ recovers the Higgs factorization
\[
M' = M_0 \overset{\mathcal K_1}\twoheadleftarrow M_1 \overset{\mathcal K_2}\twoheadleftarrow \cdots \overset{\mathcal K_c}\twoheadleftarrow M_c = M
\]
of $f$ by specifying the modular cuts to be
\[
\mathcal K_i = \{G \in \mathscr L_{M_i} \mid G \supseteq F \textnormal{ for some } F\in \operatorname{cyc}(f) \textnormal{ with } n_f(F) \geq i\}.
\]
Thus, since by assumption the $f$-cyclic flats form a chain, the descriptions of the modular cuts $\mathcal K_i$ imply that the matroid $M'$ is obtained by a sequence of principal truncations, first by $F_k$ repeated $a_k$ times, then by $F_{k-1}$ repeated $a_{k-1}$ times, and so forth,
 where $a_j = n_f(F_j) - n_f(F_{j-1})$ for $j > 1$, and $a_1 = n_f(F_1)$.
\end{proof}

Moreover, the bijection given in the previous theorem respects linear independence.

\begin{prop}\label{prop:linindep}
The elements  
\[
\{\Delta_{M'} \ : \ M' \text{ is a loopless relative nested quotient of $M$}\}
\]
are linearly independent in $\operatorname{MW}_\bullet(\Sigma_{A_n})$.
\end{prop}

The proof given below is a modification of the one given for nested matroids in \cite[Proposition 3.2]{Ham17}.

\begin{proof}
\newcommand{\M}{\mathscr M} 
Fix a matroid $M$, and let 
\[ \M = \{ \Delta_{M'} : M' \text{ is a rank $r$ loopless relative nested quotient of } M \}. \]
Write $\emptyset = F_0 \subsetneq F_1 \subsetneq \cdots \subsetneq F_k$ for the cyclic flats of the quotient $M' \twoheadleftarrow M$, and define $\gamma: \M \to \mathbb N^r$ by 
\[
    \gamma(\Delta_{M'})_i := \begin{cases} \rk_{M'}(F_i) - \rk_{M'}(F_{i-1}) & \text{if } 1 \leq i \leq k \\
        0 & \text{otherwise.}
    \end{cases}
\]
We show that there are no relations among the elements of $\M$ using a lexicographic induction.
The following claim will be used repeatedly.
\begin{claim}
Let $\Delta_{M'} \in \M$, let $\mathcal F = \{\emptyset = F_0 \subsetneq F_1 \subsetneq \cdots \subsetneq F_k\}$ be the cyclic flats of $f: M' \twoheadleftarrow M$, and let $\mathcal F'$ be a maximal chain of flats of $M'$ containing $\mathcal F$. If $\Delta_{N'} \in \M$ for a quotient $g: N'\twoheadleftarrow M$ has the property that $\mathcal F'$ is also a maximal chain of flats of $N'$, then either $M' = N'$ or $\gamma(\Delta_{N'}) <_\text{lex} \gamma(\Delta_{M'})$.
\end{claim}

\begin{claimproof}
Suppose $M'\neq N'$.  We have two cases:
\begin{enumerate}[label = (\roman*)]
\item The set of $g$-cyclic flats of $N'$ strictly contains $\mathcal F$, or
\item there exists a minimal $0 < j \leq k$ such that $F_j$ is not $g$-cyclic in $N'$.
\end{enumerate}
In the case (i), suppose $F$ is a $g$-cyclic flat of $N'$ not in $\mathcal F$.  Then, because the corank $\operatorname{rk}(M) - \operatorname{rk}(N')$ equals the sum of the consecutive differences of $g$-nullities of the $g$-cyclic flats (by the proof of Theorem 3.3.8), we observe that $F\subsetneq F_k$ necessarily for $\operatorname{rk}(M') = \operatorname{rk}(N')$ to hold.  Thus, we have $\gamma(\Delta_{N'}) <_\text{lex} \gamma(\Delta_{M'})$ in this case.
In the case (ii),
there is a cyclic flat $G$ of $N'$ with $F_{j-1} \subset G \subsetneq F_j$ and $n_g(G) = n_g(F_j)$.
    Moreover, $G \neq F_{j-1}$ because
    \[
        n_g(G) = n_g(F_j) = n_f(F_j) > n_f(F_{j-1}) = n_g(F_{j-1}).
    \]
    Consequently, we have $\gamma(\Delta_{N'})_j < \gamma(\Delta_{M'})_j$, so $\gamma(\Delta_{N'}) <_\text{lex} \gamma(\Delta_{M'})$ by our choice of $j$.
\end{claimproof}
Now we prove the proposition. Suppose that $\sum_{\Delta_{M'} \in \M} c_{M'} \Delta_{M'} = 0$.
We show $c_{M'} =0$ for all $M'$.

If $\Delta_{M'} \in \M$ with $\gamma(\Delta_{M'})$ lex-minimal in $\gamma(\M)$, then the claim implies that there is a cone of $\Sigma_{A_n}$ on which $\Delta_{M'}$ takes value 1 and all other elements of $\M$ take value 0; therefore, $c_{M'} = 0$. 
Otherwise, $\gamma(\Delta_{M'})$ is not lex-minimal. Suppose by induction that $c_{N'} = 0$ for all $\Delta_{N'} \in \M$ with $\gamma(\Delta_{N'}) <_\text{lex} \gamma(\Delta_{M'})$. By the claim, there is a cone of $\Sigma_{A_n}$ on which $\Delta_{M'}$ takes value 1 and all elements $\Delta_{N''}$ of $\M$ such that $\gamma(\Delta_{N''}) \not<_\text{lex} \gamma(\Delta_{M'})$ take value 0; therefore, $c_{M'} = 0$, as desired.
\end{proof}

\section{The Poincar\'e duality property}\label{section:poincare}

As a first application of the simplicial presentation, we establish the Poincar\'e duality property for Chow rings of matroids in this section.  While this was established in \cite[Theorem 6.19]{AHK18} by a double induction, our proof is not inductive.

\smallskip
We review some facts about Poincar\'e duality algebras in \S\ref{subsection:poincarealgebra}, but we will only need a small portion of these facts in this section---the rest will be needed later in \Cref{section:HIT}.  We prove the Poincar\'e duality property of Chow rings of matroids and discuss some consequences in \S\ref{subsection:poincare}

\subsection{Poincar\'e duality algebras}\label{subsection:poincarealgebra} We review some general algebraic notions about Poincar\'e duality algebras.  Let $\mathbbm k$ be a field.

\begin{defn}\label{dfn:pda}
A graded finite (commutative) $\mathbbm k$-algebra $A^\bullet = \bigoplus_{i=0}^d A^i$ is a \textbf{(graded) Poincar\'e duality algebra of dimension $d$} if (i) $A^0 = \mathbbm k$,  and (ii) there exists an isomorphism $\int: A^d \overset{\sim}\to \mathbbm k$, called the \textbf{degree map} of $A^\bullet$, such that the map
\[
A^k \to \operatorname{Hom}(A^{d-k}, \kk), \quad \xi \mapsto (\zeta \mapsto \int \xi\cdot \zeta)
\]
is an isomorphism for all $0\leq k \leq d$, or equivalently, the pairing
$$A^i \times A^{d-i} \to A^d \simeq \mathbbm k,\quad (\xi,\zeta) \mapsto \int \xi\cdot\zeta$$
is a non-degenerate for all $0\leq i \leq d$. \label{it:nondegenerate}
\end{defn}

We write $(A^\bullet, \int)$ for a Poincar\'e duality algebra with a chosen degree map $\int$.  In \Cref{section:HIT}, we will often drop the degree symbol $\int$ when the context is clear.  In particular, for $\zeta\in A^1$ we will often write $\zeta^d$ to mean $\int \zeta^d$.
Two useful facts about Poincar\'e duality algebras follow.  Both are straightforward to check.

\begin{prop}\label{prop:prodpoincare} Let $(A^\bullet, \int_A)$ and $(B^\bullet, \int_B)$ be Poincar\'e duality algebras of dimension $d_A$ and $d_B$ over a common field $\mathbbm k$.
\begin{enumerate}
\item The tensor product
$(A \otimes B)^\bullet = \bigoplus_{\bullet} \big(\bigoplus_{i+j = \bullet} A^i \otimes B^j \big)$ is also a Poincar\'e duality algebra of dimension $d_A + d_B$ with degree map
\[
\int_{A\otimes B}: (A\otimes B)^{d_A +d_B} = A^{d_A} \otimes B^{d_B} \to \mathbbm k, \quad a\otimes b \mapsto \int_A a \cdot \int_B b.
\]
\item A surjection $A^\bullet \twoheadrightarrow B^\bullet$ of Poincar\'e duality algebras of the same dimension is an isomorphism.
\end{enumerate}
\end{prop}

We will use the following construction to establish that Chow rings of matroids are Poincar\'e duality algebras.

\begin{prop}\label{prop:transport} If $(A^\bullet, \int)$ is a Poincar\'e duality algebra of dimension $d$, and $f\in A^\bullet$ a homogeneous element of degree $k$.  Then the $\kk$-algebra
\[
A^\bullet / \operatorname{ann}(f), \quad \textnormal{ where } \operatorname{ann}(f) = \{a \in A^\bullet \mid af = 0\}
\]
is a Poincar\'e duality algebra of dimension $d-k$ with the \textbf{induced degree map} $\int_f$ defined by $\int_f( a+\operatorname{ann}(f)) :=  \int af$ for $a\in A^{d-k}$.
\end{prop}

\begin{proof} This is a straightforward check; see \cite[Corollary I.2.3.]{MS05} for example.\end{proof}

It will sometimes be convenient to identify elements of the ring $A^\bullet/\operatorname{ann}(f)$ to elements of the principal ideal $\langle f \rangle \subset A^\bullet$, with multiplication is by $af \cdot bf = (ab) \cdot f$.  The construction in \Cref{prop:transport} will arise in next subsection \S\ref{subsection:poincare} with $f$ being the Bergman class of a matroid.

\bigskip
The rest of this subsection on Poincar\'e duality algebras will not be needed until \Cref{section:HIT}.

\medskip
 We describe another way the construction in \Cref{prop:transport} arises in the context of Chow cohomology rings of fans.  Let $\Sigma$ be a $d$-dimensional smooth rational fan in $N_\RR$ for a lattice $N$, and let $\rho \in \Sigma(1)$ be ray.  Denote by $\overline u$ the image of $u\in N_\RR$ under the projection $N_\RR \twoheadrightarrow N_\RR/\operatorname{span}(\rho)$.  The \textbf{star of $\Sigma$ at $\rho$} is a $(d-1)$-dimensional fan in $N_\RR/\operatorname{span}(\rho)$ defined by
$$\operatorname{star}(\rho,\Sigma) := \{\overline \sigma \ | \ \textnormal{$\sigma\in \Sigma$ contains $\rho$}\}.$$
By definition of the Chow cohomology ring, one can check that there is a surjection $A^\bullet(\Sigma) \twoheadrightarrow A^\bullet(\operatorname{star}(\rho,\Sigma))$ determined by
$$x_{\rho'} \mapsto
\begin{cases}
x_{\overline{\rho'}} &\textnormal{ if $\rho'$ and $\rho$ form a cone in $\Sigma$}\\
0 &\textnormal{ otherwise}
\end{cases}\qquad \textnormal{for each $\rho' \neq \rho$}.
$$
Since $\langle x_{\rho'} \ | \ \textnormal{$\rho'$ and $\rho$ do not form a cone in $\Sigma$}\rangle \subset \operatorname{ann}_{A^\bullet(\Sigma)}(x_\rho)$, thus we get an induced map
$$\pi_\rho: A^\bullet(\operatorname{star}(\rho,\Sigma))\twoheadrightarrow A^\bullet(\Sigma)/\operatorname{ann}(x_\rho).$$


\medskip
In \Cref{section:HIT}, we will use the following criterion for when the map $\pi_\rho$ is an isomorphism.

\begin{prop}\label{prop:starann}\cite[Proposition 7.13]{AHK18}
Suppose that the Chow cohomology ring $A^\bullet(\Sigma)$ is a Poincar\'e duality algebra.  Then, the map $\pi_\rho: A^\bullet(\operatorname{star}(\rho,\Sigma))\twoheadrightarrow A^\bullet(\Sigma)/\operatorname{ann}(x_\rho)$ is an isomorphism if and only if $A^\bullet(\operatorname{star}(\rho,\Sigma))$ is a Poincar\'e duality algebra.
\end{prop}

\begin{proof} The algebra $A^\bullet(\Sigma)/\operatorname{ann}(x_\rho)$ is a Poincare duality algebra by \Cref{prop:transport}.  The statement thus follows from \Cref{prop:prodpoincare}.(2).
\end{proof}

\subsection{Poincar\'e duality for matroids}\label{subsection:poincare} We show that the Chow ring $A^\bullet(M)$ of a loopless matroid $M$ is a Poincar\'e duality algebra with $\int_M$ as the degree map.  While this was proved in \cite[Theorem 6.19]{AHK18}, we give a non-inductive proof by using the simplicial presentation.

\medskip
Our main theorem of the section is the following.

\begin{thm}\label{thm:transport}
Let $M$ be a loopless matroid of rank $r=d+1$ on a ground set $E = \{0,1,\ldots, n\}$, and consider the Bergman class $\Delta_M \in \operatorname{MW}_d(\Sigma_{A_n})$ as an element of $A^\bullet(\Sigma_{A_n})$ via the isomorphism $A^\bullet(\Sigma_{A_n}) \simeq \operatorname{MW}_{n-\bullet}(\Sigma_{A_n})$ in \Cref{thm:FS}.  Then, we have
\[
A^\bullet(M) \simeq A^\bullet(\Sigma_{A_n})/ \operatorname{ann}(\Delta_M).
\]
\end{thm}

Since $A^\bullet(\Sigma_{A_n})$ is a Poincare duality algebra (\Cref{thm:FS}), Proposition \ref{prop:transport} immediately implies the following corollary.

\begin{cor}\label{cor:poincare}
The Chow ring $A^\bullet(M)$ is a graded Poincar\'e duality algebra of dimension $\rk(M) - 1$ with $\int_M$ as the degree map.
\end{cor}

\begin{rem}
Because $\Sigma_{A_n}$ is a smooth projective fan, there exists a purely combinatorial proof of the Poincar\'e duality for its Chow cohomology ring via the line shelling of the fan; see \cite{McM93,FK10}.  While Bergman fans of matroids are also shellable \cite{Bjo92}, they are not complete, and the arguments of \cite{McM93,FK10} do not readily modify to give Poincar\'e duality for Chow rings of matroids.
\end{rem}

Rephrasing \Cref{cor:poincare} yields the following generalization of \cite[Corollary 3.13]{Ham17}.

\begin{cor}\label{cor:fullhampe}
For each $0\leq c\leq d$, the cap product map
$$A^c(M) \to \operatorname{MW}_{d-c}(\Sigma_M), \quad \xi \mapsto \xi\cap \Delta_M$$
is an isomorphism of $\RR$-vector spaces.  Thus, the Bergman classes of relative nested quotients form a basis of $\operatorname{MW}_\bullet(\Sigma_M)$.
\end{cor}

\begin{proof}
The first statement follows from \Cref{cor:poincare} and $\operatorname{Hom}(A^c(M),\RR) \simeq \operatorname{MW}_{d-c}(\Sigma_M)$ (Lemma \ref{lem:kronecker}).  Theorem \ref{thm:monomnested} then implies the second statement.
\end{proof}

\begin{rem}[cf.\ \ref{obs:transport}]\label{rem:geomtransportproof}
Let $Y_{\mathscr R(M)}$ be the wonderful compactification of a realization of $M$.  Since the pullback map $A^\bullet(X_{\Sigma_{A_n}}) \to A^\bullet(Y_{\mathscr R(M)})$ along the closed embedding $Y_{\mathscr R(M)}\hookrightarrow X_{A_n}$ is surjective (\Cref{rem:matgeom}), we have $A^\bullet(M) \simeq A^\bullet(Y_{\mathscr R(M)}) \simeq A^\bullet(X_{\Sigma_{A_n}})/\operatorname{ann}([Y_{\mathscr R(M)}])$.
\end{rem}

\begin{proof}[Proof of \Cref{thm:transport}]
Recall that the isomorphism $A^\bullet(\Sigma_{A_n}) \simeq \operatorname{MW}_{n-\bullet}(\Sigma_{A_n})$ makes the set of Minkowski weights into a graded ring, denoted $\operatorname{MW}^{\bullet}(\Sigma_{A_n})$.  Let $\iota_M^*$ be the pullback map of the inclusion $\iota_M: \Sigma_M \hookrightarrow \Sigma_{A_n}$.  A formal property of cap products, given below in \Cref{lem:formalcap}, gives us a commuting diagram of surjections
\[
\begin{tikzcd}
&A^\bullet({\Sigma_{A_n}}) \arrow[r, "\sim"] \arrow[d, "\iota_M^*"', two heads]& \operatorname{MW}^\bullet(\Sigma_{A_n}) \arrow[d, two heads] & \\
&A^\bullet(M) \arrow[r, two heads] &\operatorname{MW}^\bullet(\Sigma_{A_n})/\operatorname{ann}(\Delta_M)
\end{tikzcd} 
\qquad
\begin{tikzcd}
\xi \ar[r,mapsto] \ar[d,mapsto] & \xi \cap \Delta_{\Sigma_{A_n}} \ar[d,mapsto] \\
\iota_M^* \xi \ar[r,mapsto] & \xi \cap \Delta_M.
\end{tikzcd}
\]
Here, we have identified the elements of $\operatorname{MW}^\bullet(\Sigma_{A_n})/\operatorname{ann}(\Delta_M)$ with the elements of the principal ideal $\langle \Delta_M\rangle \subset \operatorname{MW}^\bullet(\Sigma_{A_n})$, which is generated by Minkowski weights of the form $\xi \cap \Delta_M$ where $\xi \in A^\bullet(M)$.  \Cref{prop:linindep} then states that the bottom horizontal map preserves linear independence, and hence is injective as well.
\end{proof}

\begin{lem}\label{lem:formalcap}
Let $\iota^*$ be the pullback map of an inclusion of fans $\iota: \Sigma' \hookrightarrow \Sigma$ where $\Sigma$ is complete, and let $\delta_{\Sigma}: A^\bullet(\Sigma) \overset\sim\to \operatorname{MW}^\bullet(\Sigma)$ be the isomorphism in \Cref{thm:FS}.  Suppose $\Delta$ is a Minkowski weight on $\Sigma$ whose support $|\Delta|$ is contained the support $|\Sigma'|$.  Then we have a diagram
\[
\begin{tikzcd}
&A^\bullet({\Sigma}) \arrow[r, "\sim"] \arrow[d, "\iota^*"', two heads]& \operatorname{MW}^\bullet(\Sigma) \arrow[d, two heads] & \\
&A^\bullet(\Sigma') \arrow[r, two heads] &\operatorname{MW}^\bullet(\Sigma)/\operatorname{ann}(\Delta)
\end{tikzcd} 
\qquad
\begin{tikzcd}
\xi \ar[r,mapsto] \ar[d,mapsto] & \xi \cap \Delta_{\Sigma} \ar[d,mapsto] \\
\iota^* \xi \ar[r,mapsto] & \xi \cap \Delta.
\end{tikzcd}
\]
\end{lem}

\begin{proof}
We need show that the kernel $\langle x_\rho \mid \rho \in \Sigma(1)\setminus \Sigma'(1)\rangle \subset A^\bullet(\Sigma)$ of the pullback map $\iota^*$ is contained in the kernel of the map $A^\bullet(\Sigma) \overset\sim\to \operatorname{MW}^\bullet(\Sigma) \twoheadrightarrow \operatorname{MW}^\bullet(\Sigma)/\operatorname{ann}(\Delta)$.  Since $|\Delta| \subseteq |\Sigma'|$, we may consider $\Delta$ as a Minkowski weight on $\Sigma'$, and thus by functoriality of the cap product (\Cref{rem:functorial}), we have $x_\rho \cap \Delta = \iota^* x_\rho \cap \Delta = 0 \cap \Delta = 0$ for $x_\rho \in A^\bullet(\Sigma)$ where $\rho \in \Sigma(1)\setminus \Sigma'(1)$.
\end{proof}


\section{Log-concavity of the volume polynomial}\label{section:volumepolynomial}

A presentation of a graded Poincar\'e duality algebra $A^\bullet$ can be encoded via the Macaulay inverse system into a single polynomial $VP_A$, called the volume polynomial of $A^\bullet$ \cite[\S21.2]{Eis95}.  In geometric contexts, the volume polynomial takes on an additional meaning:
\begin{itemize}
\item If $A^\bullet$ is the ring of algebraic cycles modulo numerical equivalence  on a smooth projective variety, then $VP_A$ measures degrees of ample divisors (see \cite{ELMNP05}).
\item If $A^\bullet$ is the Chow cohomology ring of a complete smooth fan $\Sigma$, then $VP_A$ measures the volumes of polytopes whose normal fans coarsen $\Sigma$ (see \cite[\S13]{CLS11}).
\end{itemize}
In both geometric contexts, the volume polynomial of $A^\bullet$ is positive and log-concave on the ample cone when considered as a function $A^1 \to \RR$.

\medskip
In this section, we give a combinatorial formula for the volume polynomial $VP_M^\nabla$ of the Chow ring $\myChow^\bullet(M)$ of a loopless matroid $M$.  We show that, as in the geometric cases, the volume polynomial  $VP_M^\nabla$ when regarded as a function $\myChow^1(M) \to \RR$ is both positive and log-concave on a subcone $\mathscr K_M^\nabla$ of the ample cone $\mathscr K_M$ generated by the nontrivial simplicial generators.  While the results of \cite{AHK18} imply that the volume polynomial of a matroid satisfies such properties, we give here an independent proof by establishing that $VP_M^\nabla$ is a Lorentzian polynomial as defined in \cite{BH20}.  In the next section \S\ref{section:HIT}, we build upon the results of this section to conclude that $VP_M^\nabla$ is both positive and log-concave on the entire ample cone $\mathscr K_M$.

\subsection{Volume polynomials and Lorentzian polynomials}\label{subsection:Lorentzian} Here we review the notion of volume polynomials and how they generalize to Lorentzian polynomials.

\medskip
One can encode a graded Poincar\'e duality algebra into a single polynomial called the volume polynomial as follows.

\begin{defn}
Let $(A^\bullet, \int)$ be a graded Poincar\'e duality algebra of dimension $d$ that is generated in degree 1, with a chosen presentation $A^\bullet = \mathbbm k [x_1, \ldots, x_s]/I$ and a degree map $\int: A^d \to \mathbbm k$.  Then its \textbf{volume polynomial} $VP_A$ is a multivariate polynomial in $\mathbbm k[t_1, \ldots, t_s]$ defined by
$$VP_A(t_1, \ldots, t_s) := \int (t_1x_1 + \cdots + t_sx_s)^d$$
where we extend the degree map $\int$ to $A[t_1,\ldots, t_s] \to \mathbbm k [t_1, \ldots, t_s]$.
\end{defn}

If $(A^\bullet, \int)$ is a Poincar\'e duality algebra with a chosen presentation $A^\bullet = \mathbbm k[x_1, \ldots, x_s]/I$, then the defining ideal $I$ can be recovered from the volume polynomial $VP_A$ as follows \cite[Lemma 13.4.7]{CLS11}
\[ I = \{f(x_1, \ldots, x_s) \in \mathbbm k[x_1, \ldots, x_s] \ | \textstyle \ f(\frac{\partial}{\partial t_1}, \ldots, \frac{\partial}{\partial t_s}) \cdot VP_A(t_1, \ldots, t_s) = 0\}. \]

\medskip
In \cite{BH20}, the authors define Lorentzian polynomials as a generalization of volume polynomials in algebraic geometry and stable polynomials in optimization.
Here we briefly summarize the relevant results.

\begin{defn} A homogeneous polynomial $f\in \RR[x_1, \ldots, x_n]$ of degree $d$ is \textbf{strictly Lorentzian} if its support consists of all monomials in $x$ of degree $d$, all of its coefficients are positive, and any of its $(d-2)$nd order partial differentiation $\partial_{i_1} \cdots \partial_{i_{d-2}} f$ has Hessian matrix with Lorentzian signature $(+,-,-,\ldots, -)$.  \textbf{Lorentzian polynomials} are polynomials that can be obtained as a limit of strictly Lorentzian polynomials.
\end{defn}

To characterize Lorentzian polynomials, we need a combinatorial notion which generalizes the exchange axiom for matroids:  a collection of points $J\subset \ZZ_{\geq 0}^n$ is \textbf{M-convex} if for any $\alpha,\beta\in J$ and $i\in [n]$ with $\alpha_i > \beta_i$ there exists $j\in [n]$ such that $\alpha_j< \beta_j$ and $\alpha-e_i+e_j \in J$.  When the elements of $J$ all have the same coordinate sum, this is equivalent to stating that the convex hull of $J$ is a generalized permutohedra \cite[Theorem 4.15]{Mur03}.

\medskip
The following characterization can be seen as a linear algebraic abstraction of the proof of Teissier-Khovanskii inequalities via the Hodge index theorem for algebraic surfaces.

\begin{thm}\label{thm:Lorentzian}\cite[Theorem 5.1]{BH20} A homogeneous polynomial $f\in \RR[x_1, \ldots, x_n]$ of degree $d$ with nonnegative coefficients is Lorentzian if and only if the following two conditions are satisfied:
\begin{enumerate}
\item The support of $f$ is M-convex, and
\item The Hessian matrix of $\partial_{i_1} \cdots \partial_{i_{d-2}} f$ has at most one positive eigenvalue for any choice of $(d-2)$nd order partial differentiation.
\end{enumerate}
\end{thm}

The Lorentzian property is preserved under nonnegative linear change of variables:

\begin{prop}\label{prop:opLorentzian} \cite[Theorem 2.10]{BH20} If $f\in \RR[x_1, \ldots, x_n]$ is Lorentzian, then so is $f(Ax) \in \RR[x_1, \ldots, x_m]$ for any $n\times m$ matrix $A$ with non-negative entries.
\end{prop}

Applications to log-concavity phenomena in combinatorics arise from the following properties of Lorentzian polynomials.

\begin{thm}\label{thm:logconc}
Let $f \in\RR[x_1, \ldots, x_n]$ be a homogeneous polynomial with nonnegative coefficients. The Lorentzian property of $f$ can be characterized via log-concavity properties as follows.
\begin{enumerate}
\item\cite[Theorem 5.3]{BH20} A homogeneous polynomial $f$ is Lorentzian if and only if $f$ is strongly log-concave, in the sense that if $g$ is any partial derivative of $f$ of any order, then either $g$ is identically zero or $\log g$ is concave on the positive orthant $\RR^n_{>0}$.
\item\cite[Example 5.2]{BH20} If $n=2$, so that $f = \sum_{k=0}^d a_k x_1^{k}x_2^{d-k}$, then $f$ is Lorentzian if and only if $(a_0, a_1, \ldots, a_d)$ has no internal zeroes and is ultra log-concave, that is, $$a_{k_1}a_{k_3} \neq 0 \implies a_{k_2} \neq 0 \textnormal{ for all $0\leq k_1 < k_2 < k_3 \leq d$, and }$$
$$\frac{a_k^2}{{d \choose k}^2} \geq \frac{a_{k-1} a_{k+1}}{{ d \choose k-1}{d \choose k+1}} \textnormal{ for all $0<k<d$}.$$
\end{enumerate}
\end{thm}

We remark that (strictly) Lorentzian polynomials arise in classical algebraic geometry whenever one has a set of nef (ample) divisors on a smooth projective variety.  

\begin{rem}\label{rem:geomlorentzian}
Let $\{D_1, \ldots, D_s\}$ be nef (ample) divisors on a smooth projective $\mathbbm k$-variety $X$ of dimension $d$, and $A(X)$ its Chow ring.  Let $\int_X: A^d(X) \to \RR$ be the degree map obtained as the pushforward map along the structure map $X \to \operatorname{Spec} \mathbbm k$.  Then
$$\operatorname{vol}_X\big(\sum_{i=1}^s t_iD_i \big) := \lim_{q \to \infty} \frac{\dim_{\mathbbm k} H^0(q\sum_i t_iD_i)}{q^d/d!} = \int_X \big(\sum_i t_iD_i \big)^d$$
is a (strictly) Lorentzian polynomial \cite[Corollary 1.6.3.(iii)]{Laz04} or \cite[Theorem 10.1]{BH20}.
\end{rem}

\subsection{The \DHRtext formula} 
We prepare our formula for the volume polynomial of $\myChow^\bullet(M)$ by describing the combinatorial notions in \cite{Pos09} that we generalize to arbitrary matroids.

\medskip
We first recall Hall's marriage theorem and Rado's generalization; for proofs we point to \cite[\S11.2]{Oxl11}.  
Let $E = [n] = \{0,1,\ldots, n\}$.  A \textbf{transversal} of a collection $\{A_0, \ldots, A_m\}$ (repetitions allowed) of subsets of $E$ is a subset $I \subseteq E$ such that there exists a bijection $\phi: \{A_0, \ldots, A_m\} \to I$ satisfying $\phi(A_i) \in A_i$ for all $0 \leq i \leq m$.


\medskip
The well-known Hall's marriage theorem \cite{Hal35} states that a transversal of $\{A_0, \ldots, A_n\}$ exists if and only if  $\left| \bigcup_{j \in J} A_i \right| \geq |J|$ for all $J \subseteq [n]$.  The following theorem of Rado gives a matroid generalization of the condition given in Hall's theorem.

\begin{thm}[Rado's theorem  \cite{Rad42}] \label{thm:HallRado}
        Let $M$ be a matroid on $E$. A family of subsets $\{A_0, \ldots, A_m\}$ of $E$ has a transversal $I \subseteq E$ that is independent in $M$ if and only if
        \[
                \rk_M \Big( \bigcup_{j \in J} A_j  \Big) \geq |J|, \quad \forall J \subseteq [m].
        \]
\end{thm}
Hall's condition can be recovered from Rado's by setting $M = U_{|E|,E}$ and $m = n$.  See \cite[Theorem 11.2.2]{Oxl11} for more information and a proof of Rado's theorem.  The following variant of Hall's marriage theorem was investigated by Postnikov as a combinatorial interpretation of a formula for volumes of generalized permutohedra \cite[\S5, \S9]{Pos09}.

\begin{prop}[Dragon marriage condition]
Let $\{A_1, \ldots, A_n\}$ be a collection of subsets of $E = \{0,1,\ldots, n\}$.  There is a transversal $I \subseteq E \setminus \{e\}$ of $\{A_1, \ldots, A_{n}\}$ for every $e \in E$ if and only if
        \[
            \Big| \bigcup_{j \in J} A_j \Big| \geq |J| + 1, \quad \forall \emptyset \subsetneq J \subseteq \{1,2,\ldots, n\}.
        \]
\end{prop}
The dragon marriage theorem above follows easily from the original Hall's marriage theorem, and conversely, one can obtain Hall's marriage theorem from the dragon marriage theorem as follows: given $A_0, \ldots, A_n \subseteq E$ as in Hall's theorem, set $E' = E \sqcup \{*\}$ and $A_i' := A_0 \sqcup \{*\}$ for each $0 \leq i \leq n$ and apply Postnikov's theorem to $\{A_0', \ldots, A_n'\}$.

\medskip
We now consider a variant of Rado's theorem in the same spirit.

\begin{prop}[\DHRtext condition]\label{prop:DHRmatching}
        Let $M$ be a matroid on $E$, and $\{A_1, \ldots, A_m\}$ a collection of subsets of $E$.
        There is for every non-loop $e \in E$ an independent transversal $I\subseteq E\setminus\{e\}$ of $\{A_1, \ldots, A_m\}$ such that $I \cup e$ is independent if and only if 
        \[
                \rk_M\Big( \bigcup_{j \in J} A_j \Big) \geq |J| + 1, \quad \forall \emptyset \subsetneq J \subseteq \{1, \ldots, m\}
        \]
        and when this condition is satisfied, we say that $\{A_1, \ldots, A_m\}$ satisfy the \textbf{dragon Hall-Rado condition of $M$}, or $\DHR(M)$ for short.
\end{prop}

\begin{proof}
If $e \in E$ is not a loop, then the dragon Hall-Rado condition for a collection $\{A_1, \ldots, A_m\}$ is equivalent to Rado's condition in Theorem \ref{thm:HallRado} for the collection $\{\{e\}, A_1 \setminus \{e\}, \ldots, A_m \setminus \{e\}\}$.
Independent transversals $I'$ of the latter collection correspond bijectively via $I' \longleftrightarrow I = I' \setminus e$ to independent transversals $I \subset E \setminus \{e\}$ of the former collection such that $I \cup \{e\}$ is independent.
\end{proof}

The proof of the dragon Hall-Rado condition demonstrates that Rado's theorem is equivalent to the dragon Hall-Rado theorem by an argument analogous to how Hall's marriage theorem is obtained from dragon marriage theorem.
In summary, the combinatorics introduced in this subsection thus far can be schematically laid out as follows with the indicated logical implications:
\begin{center}
        \begin{tikzcd}
                \textrm{Hall's marriage theorem} \ar[r, Leftarrow] \ar[d, Leftrightarrow] & \textrm{Rado's theorem} \ar[d, Leftrightarrow] \\
                \textrm{Dragon marriage theorem} \ar[r, Leftarrow] & \textrm{Dragon Hall-Rado theorem}
        \end{tikzcd}
\end{center}

We are now ready to compute the intersection numbers of the variables $h_F$ in the simplicial presentation of the Chow ring of a matroid $M$.

\begin{thm}
\label{thm:dhr}
Let $A_1, \ldots, A_d$ be a collection of subsets of $E$, and $M$ a loopless matroid on $E$ of rank $d+1$.  Let $H_{A_1}, \ldots, H_{A_d}$ be matroids as defined in Proposition \ref{lem:multquot}.  Then 
$$M\wedge H_{A_1} \wedge \cdots \wedge H_{A_d} = U_{1,E} \iff  \{A_1, \ldots, A_d\} \textnormal{ satisfies $\DHR(M)$}.$$
Thus, we have 
\[ \int_M h_{A_1}(M) \cdots h_{A_d}(M) = \begin{cases} 1& \textrm{if  $\{A_1, \ldots, A_d\}$ satisfies $\DHR(M)$} \\ 0& \textrm{otherwise.} \end{cases} \]
\end{thm}

\begin{proof}
For the first assertion, we begin by making two easy observations.
\begin{enumerate}
\item $M\wedge H_S$ has a loop if and only if $\operatorname{rk}_M(S) = 1$, and 
\item \cite[Exercise 7.3.10]{Oxl11} for the elementary quotient $f: M \twoheadrightarrow M\wedge H_S$ we have by Proposition \ref{prop:quotbases} that $\{T\subseteq E \ | \ n_f(T) = 1\} =  \{T \subseteq E \ | \ \operatorname{cl}_M(T) \supseteq S\}.$  In particular, the rank of $T\subseteq E$ decreases by 1 in the elementary quotient if $T\supseteq S$.
\end{enumerate}

For the $\implies$ direction, suppose $\operatorname{rk}_M(\bigcup_{j\in J} A_j) \leq k$ for a nonempty $J = \{j_1, \ldots, j_k\} \subseteq \{1, \ldots, d\}$.  Then, for $\widetilde M := M\wedge H_{A_{j_1}} \wedge \cdots \wedge H_{A_{j_{k-1}}}$ we have $\operatorname{rk}_{\widetilde M}(\bigcup_{j\in J} A_j) \leq k - (k-1) = 1$, so that $M \bigwedge_{j\in J} H_{A_j}$ already has a loop.

For the $\impliedby$ direction, we induct on $d$.  The base case $d=1$ is easily verified.  Now, we claim that if $\{A_1, \ldots, A_d\}$ satisfies the dragon Hall-Rado condition for $M$, then so does $\{A_1, \ldots, A_{d-1}\}$ for $\widetilde M := M \wedge H_{A_d}$.  For the sake of contradiction, suppose $\{A_1, \ldots, A_{d-1}\}$ fails $\DHR(\widetilde M)$.  That is, without loss of generality suppose that $\operatorname{rk}_{\widetilde M}(A_1 \cup \cdots \cup A_k) \leq k$ for some $1\leq k \leq d-1$.  Then, we must have had $\operatorname{rk}_M(A_1\cup \cdots \cup A_k) = k+1$ with $\operatorname{cl}_M(A_1 \cup \cdots \cup A_k) \supseteq A_d$.  But then $\operatorname{rk}_M(A_1\cup \cdots \cup A_k \cup A_d) = k+1$, violating $\DHR(M)$.

For the second assertion, we first note that $\int_M h_{A_1}(M) \cdots h_{A_d}(M) = \int_{\Sigma_{A_n}} h_{A_1} \cdots h_{A_d} \cap \Delta_M$.  Repeated application of \Cref{thm:hyper} implies that $h_{A_1} \cdots h_{A_d} \cap \Delta_M$ is either 0 or the Bergman class $\Delta_{M'}$ of a loopless matroid $M'$ of rank 1, in which case $M'$ moreover equals $M\wedge H_{A_1} \wedge \cdots \wedge H_{A_d}$.  Since $U_{1,E}$ is the unique loopless matroid  of rank 1 on $E$, which defines the Bergman class $\Delta_{U_{1,E}}$ by $\Delta_{U_{1,E}}(\mathbf 0) = 1$ where $\mathbf 0$ is the zero-dimensional cone of $\Sigma_{A_n}$, we have that $\int_{\Sigma_{A_n}} h_{A_1} \cdots h_{A_d} \cap \Delta_M$ equals 1 if $\{A_1, \ldots, A_d\}$ satisfies $\DHR(M)$ and equals 0 otherwise.
\end{proof}

We obtain as an immediate corollary the promised generalization of \cite[Corollary 9.4]{Pos09}.  Recall that $\mathscr L_M^{\geq 2}$ denotes the flats of $M$ of rank at least two.

\begin{cor}\label{cor:volpolH}
Let $M$ be a loopless matroid on $E$ of rank $d+1$.  The volume polynomial $VP_M^\nabla(\underline t)\in \QQ[t_F \ | \ F\in \mathscr L_M^{\geq 2}]$ of $\myChow^\bullet(M)$ is
$$VP_M^\nabla(\underline t) = \sum_{(F_1, \ldots, F_d)} t_{F_1}\cdots t_{F_d}$$
where the sum is over ordered collections of nonempty flats $F_1, \ldots, F_d$ of $M$ satisfying $\DHR(M)$.  Alternatively, we have
$$VP_M^\nabla(\underline t) = \sum_{\{F_1^{d_1}, \ldots, F_k^{d_k}\}} {d \choose d_1, \ldots, d_k} t_{F_1}^{d_1}\cdots t_{F_k}^{d_k}$$
where the sum is over size $d$ multisets $\{F_1^{d_1}, \ldots, F_k^{d_k}\}$ of flats of $M$ satisfying $\DHR(M)$.
\end{cor}

One recovers the following central result of \cite{Pos09} by setting $M = U_{|E|,E}$.

\begin{cor}\label{cor:pos}\cite[Corollary 9.4]{Pos09}
The volume polynomial $VP^\nabla_{U_{n+1,n+1}}(\underline t)$ of $\myChow^\bullet(X_{A_n})$ is
$$VP^\nabla_{U_{n+1,n+1}}(\underline t) = \sum_{(S_1, \ldots, S_n)} t_{S_1}\cdots t_{S_n}$$
where the sum is over ordered collections of nonempty subsets $S_1, S_2, \ldots, S_n$ such that $|\bigcup_{j\in J} S_j| \geq |J|+1$ for any $\emptyset \subsetneq J\subseteq \{1, \ldots, n\}$.
\end{cor}

The volume polynomial $VP_M$ of the more classical presentation $A^\bullet(M)$ of the Chow ring of a matroid $M$ by generators $\{x_F \ | \ F\in \mathscr L_M\setminus\{\emptyset, E\}\}$ was computed in \cite{Eur20}.

\subsection{Volume polynomial of a matroid is Lorentzian}  
Motivated by Remark \ref{rem:geomlorentzian} and the fact that the simplicial generators are combinatorially nef, we prove here that the volume polynomial $VP_M^\nabla$ of the simplicial presentation $\myChow^\bullet(M)$ is Lorentzian.

\begin{thm}\label{thm:VPLorentzian}
The volume polynomial $VP_M^\nabla \in \RR[t_F \ | \ F\in \mathscr L_M^{\geq 2}]$ of a loopless matroid $M$ is Lorentzian.
\end{thm}

As an immediate corollary, by applying Theorem \ref{thm:logconc} to Theorem \ref{thm:VPLorentzian} we obtain:

\begin{cor}\label{cor:VPlogconc}
The volume polynomial $VP_M^\nabla$, as a polynomial in $\RR[t_F \ | \ F\in \mathscr L_M^{\geq 2}]$, is strongly log-concave in the positive orthant $\RR_{>0}^{\mathscr L_M^{\geq 2}}$.  In other words, as a function $A^1(M) \to \RR$, the polynomial $VP_M^\nabla$ is strongly log-concave in the interior of the cone $\mathscr K_M^\nabla$ generated by the nontrivial simplicial generators.
\end{cor}

We will show that the volume polynomial $VP_M^\nabla$ of a loopless matroid $M$ satisfies the two conditions listed in Theorem \ref{thm:Lorentzian}.  First, we see that the dragon Hall-Rado condition description of the support of $VP_M^\nabla$ implies that $VP_M^\nabla$ has M-convex support.

\begin{prop}\label{prop:mconvex}
Let $\{F_1, \ldots, F_d\}$ and $\{G_1, 
\dots, G_d\}$ be two multisets of flats of $M$ such that both $t_{F_1}\cdots t_{F_d}$ and $t_{G_1}\cdots t_{G_d}$ are in the support of $VP_M^\nabla$.  Without loss of generality, suppose  $G_d$ is a flat which appears more times in $\{G_1, 
\dots, G_d\}$ than it does in $\{F_1, \dots, F_d\}$.  Then, there exists another flat $F_m$ which appears more times in $\{F_1, \ldots, F_d\}$ than it does in $\{G_1, 
\dots, G_d \}$ such that $t_{F_1}\cdots t_{F_d}t_{G_d}/t_{F_m}$ is in the support of $VP_M^\nabla$.
\end{prop}

\begin{proof}  First note that$\{F_1, \dots, F_d\}$ and $\{G_1, 
\dots, G_d\}$ satisfy the Dragon-Hall-Rado condition.  In this argument, we borrow standard language from (poly)matroid theory \cite{HH02}.  Let us call a multiset of flats $\{A_1, \dots A_k\}$ ${\emph dependent}$ if $\operatorname{rk}_M(\bigcup_{j=1}^{k} A_j) \leq k  $, and ${\emph independent}$ otherwise.  We will also call a minimally dependent set of flats a \emph{circuit}. We claim that the multiset of flats $\{F_1, \ldots, F_d, G_d\}$ contains a unique circuit $X$.  The proposition will follow from this claim because we can take $F_m$ to be any flat in $X$ which appears more times in  $\{F_1, \ldots, F_d\}$ than it does in $\{G_1, \dots, G_d\}$, and $\{F_1, \ldots, F_d, G_d\}\setminus\ \{F_m\}$  will satisfy the Dragon-Hall-Rado condition.

To prove the claim, suppose to the contrary that $\{R_1, \dots, R_a\}$ and $  \{S_1, \dots, S_b\}$ are two distinct circuits which are both subsets of $\{F_1, \dots F_d, G_d\}$.  We will derive a contradiction by producing a dependent set of flats contained in $\{F_1, \ldots, F_d\}$.   By assumption, $\operatorname{rk}_M(\bigcup_{j=1}^{a} R_j) \leq a$, and for each $1\leq i\leq a$, we have $\operatorname{rk}_M(\bigcup_{j\neq i} R_j) \geq a$, therefore $\operatorname{rk}_M(\bigcup_{j=1}^{a} R_j) = a$.  Similarly, $\operatorname{rk}_M(\bigcup_{j=1}^{b} S_j) =b$.  Because $\{R_1, \dots, R_a\}$ and $\{S_1, \dots, S_b\}$ are not fully contained in $\{F_1,\dots, F_d\}$, we know that $G_d \in \{R_1, \dots, R_a\} \cap \{S_1, \dots, S_b\}$. Let $\{R_1, \dots, R_a\} \cap \{S_1, \dots, S_b\} = \{T_1, \dots, T_c\}$, where intersection is multiset intersection, and without loss of generality assume $G_d = R_a = S_b = T_c$.

Let $R$ and $S$ be the joins of the elements in $\{R_1, \dots, R_a\}$ and $\{S_1, \dots, S_b\}$, respectively.  We have that $R = \bigvee_{j=1}^{a-1} R_j =\bigvee_{j=1}^{a} R_j$ and $S = \bigvee_{j=1}^{b-1} S_j =\bigvee_{j=1}^{b} S_j$, otherwise $\{R_1, \dots, R_{a-1}\}$ and $\{S_1, \dots, S_{b-1}\}$ would both be dependent, contradicting minimality.  Therefore the join of the elements in the multiset 
 $\{R_1, \dots, R_{a-1},S_1, \dots, S_{b-1}\} \setminus \{T_1, \dots T_{c-1}\}$ is $R \vee S$.  Note that the previous subtraction is multiset subtraction.   Furthermore, $|\{R_1, \dots, R_{a-1},S_1, \dots, S_{b-1}\} \setminus \{T_1, \dots T_{c-1}\}| =  (a-1)+(b-1)-(c-1) = a+b-c-1.$

 By minimality of $\{R_1, \dots, R_a\}$ and $\{S_1, \dots, S_b\}$, we know that $\{T_1, \dots, T_c\}$ is independent, thus $\operatorname{rk}_M(\bigcup_{j=1}^{c} T_j) \geq c+1  $.  Submodularity then gives that $$\operatorname{rk}_M(R \cup S)\leq \operatorname{rk}_M(R)+ \operatorname{rk}_M( S)-\operatorname{rk}_M(R \cap S)$$ $$  \leq a+b -\operatorname{rk}_M(\bigcup_{j=1}^{c} T_j) \leq a+b-c-1. $$  
 
 Therefore the mulitset $\{R_1, \dots, R_{a-1},S_1, \dots, S_{b-1}\} \setminus \{T_1, \dots T_{c-1}\}$ is dependent in $\{F_1, \ldots, F_d\}$, which contradicts the fact that this set satisfied the Dragon-Hall-Rado condition.
\end{proof}

\begin{rem}\label{rem:prodproj}
Suppose $M$ has a realization $\mathscr R(M)$.  By \Cref{rem:wndcpt}.(2), the wonderful compactification $Y_{\mathscr R(M)}$ is embedded in the product $\prod_{F \in \mathscr L_M \setminus \{\emptyset\}} \PP(V^*/L_F)$.
Our simplicial generators are pullbacks of the hyperplane classes of the projective spaces $\PP(V^*/L_F)$ (see \Cref{rem:simplicialpullback}).  Thus, in this case, that the support of $VP_M^\nabla$ is M-convex follows from the result of \cite{CCRLMZ20} that the multidegree of an irreducible variety in a product of projective spaces has M-convex support.
\end{rem}

\begin{proof}[Proof of Theorem \ref{thm:VPLorentzian}]
Let $M$ be a loopless matroid of rank $r= d+1$.  There is nothing to prove if $d = 1$, so we assume $d\geq 2$.  The support of $VP_M^\nabla$ is M-convex by the previous proposition.  We now show that the quadrics obtained as the $(d-2)$-th partial derivatives have the desired signature.  Observe that for a flat $G$ of rank $\geq2$, we have
$$\frac{\partial}{\partial t_G} VP_M^\nabla(t) = d\int_M h_G \cdot \Big(\sum_{F\in \mathscr L_M^{\geq 2}} t_F h_F \Big)^{d-1} = d\int_{T_G(M)} \Big(\sum_{F\in \mathscr L_M^{\geq 2}} t_Fh_{\operatorname{cl}_{T_G(M)}(F)} \Big)^{d-1}.$$
Now, suppose $\{F_1, \ldots, F_{d-2}\}$ is a multiset of size $d-2$ consisting of flats of $M$ with rank at least 2.  If $\{F_1, \ldots, F_{d-2}\}$ does not satisfy $\DHR(M)$, then $\partial_{t_{F_1}}\cdots \partial_{t_{F_{d-2}}} VP_M^\nabla \equiv 0$, so we may assume that $\{F_1, \ldots, F_{d-2}\}$ satisfies $\DHR(M)$.  One computes that
$$\partial_{t_{F_1}}\cdots \partial_{t_{F_{d-2}}} VP_M^\nabla(t) = \frac{d!}{2!} \int_{M'} \Big (\sum_{F\in \mathscr L_M^{\geq 2}} t_F h_{\operatorname{cl}_{M'}(F)} \Big)^2$$
where $M' = M \wedge H_{F_1} \wedge \cdots \wedge H_{F_{d-2}}$ is a loopless matroid of rank 3.  By Proposition \ref{prop:opLorentzian}, it suffices to check that $VP_{M'}^\nabla$ is Lorentzian.  For any loopless matroid $M'$ of rank 3, the degree 1 part $\myChow^1(M)$ of its Chow ring has the simplicial basis $\{h_E\} \cup \{ h_F : \operatorname{rk}_{M'}(F) = 2\}$.  Noting that $\int_{M'} h_E\cdot h_E = 1$, $\int_{M'} h_E\cdot h_F = 1$, and $\int_{M'}h_F\cdot h_{F'} = 1$ if $F\neq F'$ and 0 otherwise, the Hessian matrix of the quadratic form $VP^\nabla_{M'}$ is 2 times the matrix
$$
\mathcal H = \begin{bmatrix}
1 & 1 & 1 & \cdots & 1\\
1 & 0 & 1 & \cdots & 1\\
1 & 1 & \ddots & \ddots & \vdots\\
\vdots & \vdots & \ddots & \ddots & 1 \\
1 & 1 & \cdots & 1 & 0
\end{bmatrix}
\textnormal{, which is congruent to}
\begin{bmatrix}
1 & \ & \ & \ \\
\ & -1 & \ & \ \\
\ & \ & \ddots & \ \\
\ & \ & \ & -1
\end{bmatrix}
$$
via the symmetric Gaussian elimination:  Explicitly, let $E_i$ be the elementary matrix such that multiplying $E_i$ to a matrix $A$ results in adding $-1$ times the first row of $A$ to the $i$-th row of $A$, and let $E_i^T$ be its transpose.  Then, the claimed congruence is given by $E_{|\mathscr L_M\setminus \{\emptyset\}|} \cdots E_3 E_2 \mathcal  H E_2^T E_3^T \cdots E_{|\mathscr L_M\setminus \{\emptyset\}|}^T$.
\end{proof}

\section{The Hodge theory of matroids in degrees at most 1}\label{section:HIT}

The \textbf{reduced characteristic polynomial} of $M$ is defined as
\[
        \bar \chi_M(t) := \frac{1}{t-1} \sum_{F \in \mathscr L_M} \mu(\emptyset, F) t^{\rk(M) - \rk(F)} = \sum_{k = 0}^{d} (-1)^k \mu^k(M) t^{d - k}
\]
where $\mu(-,-)$ is the M\"obius function of the lattice $\mathscr L_M$ and $\mu^i(M)$ is the absolute value of the $i$th coefficient of $\overline\chi_M(t)$.  The Heron-Rota-Welsh conjecture \cite{Rot71, Her72, Wel76} stated that
\[
        \mu^{k-1}(M) \mu^{k+1}(M) \leq \mu^k(M)^2 \quad \text{for $0 < k < d$}.
\]

To resolve the conjecture, the authors of \cite{AHK18} showed that the Chow ring of a matroid satisfies properties enjoyed by the cohomology ring of a smooth projective complex variety.  Namely, these are the Poincar\'e duality property, the hard Lefschetz property, and the Hodge-Riemann relations, which together form the ``K\"ahler package.''  In particular, the Hodge-Riemann relation in degree 1 implies the Heron-Rota-Welsh conjecture \cite[\S9]{AHK18}.

\smallskip
To prove that the Chow ring of a matroid satisfies the Hodge-Riemann relations, the authors of \cite{AHK18} adapt a line of argument that originally appeared in McMullen's work on simple polytopes \cite{McM93}. 
Their method employs a double induction on the rank of the matroid and the size of an order filter on the matroid's lattice of flats: the outer induction on rank shows that the Hodge-Riemann relations hold for all ample classes if they hold for a single ample class, and the inner induction on the size of an order filter is then used to construct an ample class for which the Hodge-Riemann relations can be verified.

\medskip
In this section, we independently establish the Hodge-Riemann relations in degree 1 using a similar argument.
As we have established in the previous section that the volume polynomial $VP_M^\nabla$ of a matroid $M$ is strongly log-concave in the subcone $\mathscr K_M^\nabla$ of the ample cone $\mathscr K_M$, we are able to avoid working with generalized Bergman fans associated to order filters and the flipping operation which interpolates between them.  Thus our proof involves only classical Bergman fans associated to matroids and takes the form of a single induction on rank alone.

\subsection{The K\"ahler package in degree 1 and log-concavity}

We begin by discussing here the statements of the K\"ahler package, and how in degree 1 they relate to log-concavity.  We then provide some generalities on the inductive paradigm for proving K\"ahler package for Chow cohomology rings of fans, similar to one in \cite{AHK18} adapted from the earlier work \cite{McM93}.

\begin{defn}
\label{dfn:hlhr}
Let $(A^\bullet, \int)$ be a Poincar\'e duality $\mathbbm k$-algebra of dimension $d$ with degree map $\int$.  For $\ell \in A^1$ and $0\leq i \leq \lfloor\frac{d}{2}\rfloor$, we define $L_{\ell}^k$ to be the Lefschetz operator
$$L_{\ell}^i: A^i \to A^{d-i},\ a\mapsto \ell^{d-2i}a,$$
and define $Q_\ell^i$ to be Hodge-Riemann symmetric bilinear form
$$Q_\ell^i: A^i \times A^i\to \mathbbm k, \ (x,y) \mapsto \int xy\ell^{d-2i}.$$
We define the set of \textbf{degree $i$ primitive classes of $\ell$} to be $P_{\ell}^i := \{ x \in A^k: x\ell^{d-2i+1}=0\}$. 
\end{defn}

\begin{defn}
Let $(A^\bullet, \int)$ be a Poincar\'e duality $\RR$-algebra of dimension $d$, and let $\ell \in A^1$.  For $0\leq i \leq \lfloor\frac{d}{2}\rfloor$, we say that $(A^\bullet, \int)$ satisfies 
\begin{itemize}
    \item$\operatorname{HL}^i_\ell$ if $L_{\ell}^i$ induces an isomorphism between  $A^i$ and $A^{d-i}$, and
    \item $\operatorname{HR}_\ell^i$ if the symmetric form $(-1)^i Q_{\ell}^i$ is positive-definite when restricted to $P^i_\ell$.
\end{itemize}
Moreover, for $\mathscr K$ a convex cone in $A^1$, we say that $(A^\bullet, \int, \mathscr K)$ satisfies the \textbf{hard Lefschetz property} ($\operatorname{HL}^i_{\mathscr K}$), resp.\ the \textbf{Hodge-Riemann relation} ($\operatorname{HR}^i_{\mathscr K}$), in degree $i$ if $A^\bullet$ satisfies $\operatorname{HL}_\ell^i$, resp. $\operatorname{HR}_\ell^i$, for all $\ell\in \mathscr K$.
\end{defn}

The Poincar\'e duality property (PD) of $(A^\bullet, \int)$ implies that the form $Q_\ell^i$ is non-degenerate if and only if $\operatorname{HL}_\ell^i$ holds.  The properties (PD), (HL), and (HR) together are called the \textbf{K\"ahler package} for a graded ring $A^\bullet$.  We will write $\operatorname{HL}^{\leq i}$ and $\operatorname{HR}^{\leq i}$ to mean hard Lefschetz property and Hodge-Riemann in degrees at most $i$, respectively.  The relation between log-concavity and the K\"ahler package in degree $\leq 1$ was realized in various contexts; for a survey we point to \cite{Huh18a}.  We will only need the following, adapted from \cite[Proposition 5.6]{BH20}.  It also appeared in \cite[\S2.3]{AOGV18}, and is a consequence of the Cauchy interlacing theorem.

\begin{prop}\label{prop:logconcHR}
Let $A^\bullet$ be a Poincar\'e duality $\RR$-algebra of dimension $d$ with degree map $\int$, and $\mathscr K$ a convex cone in $A^1$.  Suppose $(A^\bullet, \int, \mathscr K)$ satisfy $\operatorname{HL}_{\mathscr K}^0$ and $\operatorname{HR}_{\mathscr K}^0$.  Then the following are equivalent:
\begin{enumerate}
\item The volume function $\operatorname{vol}: A^1 \to \RR, \ \ell \mapsto \int \ell^d$ is log-concave on $\mathscr K$, and
\item for any $\ell\in \mathscr K$, the symmetric form $Q_\ell^1$ has exactly one positive eigenvalue.
\end{enumerate}
In particular, if the volume polynomial $VP_A$ of $A^\bullet = \RR[x_1, \ldots, x_s]/I$ is Lorentzian, then $(A^\bullet, \int, \mathscr K)$ satisfies $\operatorname{HR}^{\leq 1}_{\mathscr K}$ where $\mathscr K$ is the interior of $\operatorname{Cone}(x_1,\ldots, x_s)$, provided that $A^\bullet$ satisfies $\operatorname{HL}_{\mathscr K}^{\leq 1}$.
\end{prop}

\medskip
We now turn to an inductive paradigm for establishing (HL) and (HR).  We assume all Poincar\'e duality algebras to be over $\RR$.  We begin by noting an easy linear algebraic observation also made in \cite[Proposition 7.16]{AHK18}.

\begin{prop}\label{fullHIT}
Let $(A^\bullet, \int, \mathscr K)$ be a Poincar\'e duality algebra which satisfies $\operatorname{HL}^i_{\mathscr K}$ for $\mathscr K$ in a convex cone in $A^1$. Suppose that $(A^\bullet, \int)$ satisfies $\operatorname{HR}^i_{\ell}$ for some $\ell \in \mathscr K$.  Then $A^\bullet$ satisfies $\operatorname{HR}^i_{\mathscr K}$.
\end{prop}

\begin{proof}
Let $\ell' \in \mathscr K$, and let $l(t) = t\ell + (1-t)\ell'$ for $t \in [0,1]$ be a line segment connecting $\ell$ and  $\ell'$.  By convexity of $\mathscr K$, we know that every point on $l$ is in $\mathscr K$.  If the signature of the bilinear pairing $Q_{l(t)}^i$ changes along $l(t)$ starting at $\ell$, then it must degenerate at some point $l(t_0)$ for $t_0\in [0,1]$, but this violates $\operatorname{HL}_{\mathscr K}^i$.
\end{proof}

We now note how properties (HL) and (HR) behave under tensor products and quotients by annihilators of an element.  While these are adapted from \cite[\S7]{AHK18} where they are phrased in terms of Chow cohomology rings of fans, because we restrict ourselves K\"ahler package up to degree 1, we can provide here easier and more direct proofs for general Poincar\'e duality algebras.

\begin{prop}\label{HITtens}
Let $(A^\bullet, \int_A)$ and $(B^\bullet, \int_B)$ be two Poincar\'e duality algebras of dimension $d_A\geq 1$ and $d_B\geq 1$.  Suppose that $A^\bullet$ and $B^\bullet$  satisfy $\operatorname{HR}_{\ell_A}^{\leq 1}$ and $\operatorname{HR}_{\ell_B}^{\leq 1}$, respectively, then $((A \otimes B)^\bullet, \int_{A\otimes B})$ satisfies $\operatorname{HR}_{\ell_A \otimes 1 + 1 \otimes \ell_B}^{\leq 1}$.
\end{prop}

Before giving the proof, we remark that if $d_A =0$ then $(A\otimes B)^\bullet \simeq B^\bullet$ (likewise if $d_B = 0$) so that the statement in the proposition is trivially satisfied after suitable modifications.  If $d_A= 1$ (or $d_B = 1$), so that $\operatorname{HR}_{\ell_A}^1$ (or $\operatorname{HR}_{\ell_B}^1$) is vacuous, our proof below implies that $(A^\bullet \otimes B^\bullet)$ satisfies $\operatorname{HR}_{\ell_A \otimes 1 + 1 \otimes \ell_B}^{1}$ if $\operatorname{HR}_{\ell_A}^0$ and $\operatorname{HR}_{\ell_B}^0$ are satisfied.

\begin{proof}
Set $\ell := \ell_A \otimes 1 + 1 \otimes \ell_B$.  First, note that $\operatorname{HR}^0_{\ell}$ follows easily from the description of the Poincar\'e duality algebra $(A\otimes B)^\bullet$ in Proposition \ref{prop:prodpoincare}.(1).  Now, we verify that $(A \otimes B)^\bullet$ satisfies $\operatorname{HR}^1_\ell$.   Let $v_1, \dots v_m$ and $w_1, \dots w_n$ be orthonormal bases for $P_{\ell_A}^1$ and $P_{\ell_B}^1$, respectively. Then
$$A^1 \cong \bigoplus_{i=1}^m \langle v_i \rangle \oplus \langle \ell_A \rangle \quad\textnormal{and}\quad B^1 \cong \bigoplus_{i=1}^n \langle w_i \rangle \oplus \langle \ell_B \rangle.$$

Noting that $(A\otimes B)^\bullet$ is a Poincar\'e duality algebra of dimension $d = d_A + d_B$, we expand
$$\ell^{d-2}= ((\ell_A \otimes 1)+(1 \otimes \ell_B))^{d-2} = \sum_{i=0}^{d-2} {d-2 \choose i}(\ell_A^{i} \otimes \ell_B^{d-i-2}).$$
The symmetric matrix for $Q_{\ell}^1$ with respect to the above basis is given by

 \[ Q_{\ell}^1(a,b) = \begin{cases} 
          -{d-2 \choose d_A} & a= b = (v_i\otimes 1) \\
          
          -{d-2 \choose d_B} & a= b = (1 \otimes w_j) \\
          \lambda{d-2 \choose d_A-2} & a = b = (\ell_A \otimes 1) \\
          \lambda{d-2 \choose d_A} & a=b=(1 \otimes \ell_B) \\
          \lambda{d-2 \choose d_A-1} & a = (\ell_A \otimes 1) \,\,{\rm and}\,\, b= (1 \otimes \ell_B) \\
          \lambda{d-2 \choose d_B-1} & a = (1 \otimes \ell_B) \,\,{\rm and}\,\, b= (\ell_A \otimes 1) \\
          0 & a= (v_i\otimes 1) \,\,{\rm and}\,\, b = (1 \otimes w_j) \,\,{\rm or}\,\, (1 \otimes \ell_B) \\
          0 & a= (v_i\otimes 1) \, \,{\rm or}\,\, (\ell_A \otimes 1) \,\,{\rm and}\,\,   b = (1 \otimes w_j) \\
       \end{cases}
    \]
    where $\lambda := (\int_A\ell_A^{d_A})(\int_B \ell_B^{d_B})$.
    
\medskip
So the matrix $Q_{\ell}^1(a,b)$ is a block matrix comprised of 3 blocks.  By $\operatorname{HR}^1_{\ell_A}$ and $\operatorname{HR}^1_{\ell_B}$, the first two blocks are negative identity matrices induced by $\{(v_i\otimes 1)\}\times\{(v_i\otimes 1)\} $ and $\{(1 \otimes w_j)\}\times \{(1 \otimes w_j)\}$.  The third and only nontrivial block is induced by $\{(\ell_A \otimes 1), (1 \otimes \ell_B)\} \times \{(\ell_A \otimes 1), (1 \otimes \ell_B)\}$, which gives the $2\times2$ matrix $$M = \lambda \begin{bmatrix}
    {d-2 \choose d_A-2} & {d-2 \choose d_A-1} \\
    {d-2 \choose d_A-1} & {d-2 \choose d_A}
\end{bmatrix}.$$
One computes that $\det(M) < 0$, and hence $M$ has signature $(+,-)$.  We conclude that $Q_{\ell}^1(a,b)$ is nondegenerate and has exactly one positive eigenvalue completing the proof.
\end{proof}

\begin{prop}\label{localHIT}
Let $(A^\bullet = \RR[x_1, \ldots, x_s]/I, \int)$ be a Poincar\'e duality algebra of dimension $d$, and let $\ell \in A^1$ be an effective divisor---that is, a non-negative linear combination of $\{x_1, \ldots, x_s\}$.  Denote by $\ell_k$ the image of $\ell$ in $A^\bullet/\operatorname{ann}(x_k)$.  For $0\leq i \leq \lfloor \frac{d-1}{2}\rfloor$, if $(A^\bullet/\operatorname{ann}(x_k), \int_{x_k})$ satisfies $\operatorname{HR}^i_{\ell_k}$ for every $k = 1, \ldots, s$, then $(A^\bullet, \int)$ satisfies $\operatorname{HL}^i_{\ell}$.
\end{prop}

\begin{proof}
Let $\ell = \sum_{k=1}^s c_{k}x_{k}$ with $c_{k} \in \mathbb{R}_{\geq 0}$, and suppose $\ell^{d-2i}f=0$ for some $f \in A^i$.  We will show that $f =0$ necessarily.  Let $f_{i}$ the image of $f$ in $A^\bullet/\operatorname{ann}(x_i)$.  As $0 = \ell^{d-2i}f$, we have $0= \ell_{k}^{d-2i}f_{k}$, and because $A^\bullet/\operatorname{ann}(x_k)$ is a Poincar\'e duality algebra of dimension $d-1$, we conclude that $f_{k}$ belongs to the primitive space $P_{\ell_{k}}^i$.  Now, we note that for all $k$,
$$\textstyle 0 = \int \ell^{d-2i}f^2 = \int (\sum_{k}c_{k} x_k)\ell^{d-2i-1}f^2 = \sum_{k} (\int_{x_k} c_{k} \ell_{k}^{d-2i-1}f_{k}^2) \quad \textnormal{and} \quad \int_{x_k} c_k\ell_k^{d-2i-1}f_k^2 \leq 0$$
where the last inequalities follow from $\operatorname{HR}_{\ell_k}^i$.  Moreover, as $Q_{\ell_k}^1$ is negative-definite on $P_{\ell_k}^i$, we conclude each $f_{k}$ to be 0, that is, $x_k f = 0$ for all $k = 1, \ldots, s$.  Since $\{x_1, \ldots, x_s\}$ generate $A^\bullet$, the Poincar\'e duality property of $A^\bullet$ implies that if $f\neq 0$ then there exists a polynomial $g(x)$ of degree $d-i$ such that $\int g(x)f \neq 0$, and hence we conclude that $f = 0$.
\end{proof}

\subsection{K\"ahler package in degree at most one for matroids} We now specialize our discussion to the setting of matroids, and establish K\"ahler package in degree at most one for Chow rings of matroids.  As a consequence, we recover the proof of Heron-Rota-Welsh conjecture as in \cite{AHK18}.

\begin{thm}\label{thm:HRone}
The Chow ring of a matroid $(A^\bullet(M), \int_M, \mathscr K_M)$ satisfies $\operatorname{HL}_{\mathscr K_M}^{\leq 1}$ and $\operatorname{HR}_{\mathscr K_M}^{\leq 1}$.
\end{thm}

We will prove the theorem by induction on the rank of the matroid.  The key combinatorial observation that allows one to reduce the rank is the following, adapted from \cite[Proposition 3.5]{AHK18}.  It underlies the well-known Hopf algebraic structure for the lattice of flats of a matroid; see \cite{KRS99, AA17} for a detailed discussion of Hopf algebraic structures for matroids.

\begin{lem}\label{lem:bergmanprod}
Let $M$ be a loopless matroid, and $F$ a nonempty proper flat of $M$.  Let $\rho_F$ be the ray corresponding to $F$ in the Bergman fan $\Sigma_M$ of $M$.  We have
\begin{enumerate}
\item $\operatorname{star}(\Sigma_M, \rho_F) \simeq \Sigma_{M|F} \times \Sigma_{M/F}$, and consequently,
\item an isomorphism of Poincar\'e duality algebras $A^\bullet(M) / \operatorname{ann}(x_F)\simeq  (A(M|F)\otimes A(M/F))^\bullet$ such that if $\ell \in \mathscr K_M$ then its image in $A^\bullet(M)/\operatorname{ann}(x_F)$ is in $(\mathscr K_{M|F} \otimes 1) \oplus (1\otimes \mathscr K_{M/F})$.
\end{enumerate}
\end{lem}

\begin{proof}
A face of $\Sigma_M$ is in $\operatorname{star}(\Sigma_M,F)$ if and only if it corresponds to a flag of flats which contains $F$.  Any such flag naturally factors as the concatenation of two flags, one with maximal element strictly contained in $F$, and the other with minimal element $F$.  This geometrically corresponds to the factorization of fans in the statement (1).  For the second statement (2), first note that $M|F$ and $M/F$ are loopless since $F$ is a flat.  Then, combine Proposition \ref{prop:prodpoincare} and Proposition \ref{prop:starann} with the easily verifiable fact that $A^\bullet(\Sigma \times \Sigma') \simeq (A(\Sigma) \otimes A(\Sigma'))^\bullet$ for rational fans $\Sigma$ and $\Sigma'$.  This proves the isomorphism $A^\bullet(M)/\operatorname{ann}(x_F) \simeq (A(M|F) \otimes A(M/F))^\bullet$.  Lastly, the statement about $\ell \in \mathscr K_M$ follows from the fact that restriction of submodular functions on lattices remain submodular under restriction to intervals in the lattice.
\end{proof}

The remaining key part of the induction in the proof of Theorem \ref{thm:HRone}, in light of Proposition \ref{fullHIT}, is to establish $\operatorname{HR}_\ell^1$ for some divisor $\ell \in \mathscr K_M$.  In \cite{AHK18} the authors employ the method of order filters and flips for this purpose; in our case, the Lorentzian property of the volume polynomial provides the desired key step.

\begin{lem}\label{lem:lorentzianHR}
Let $M$ be a loopless matroid of rank $r = d+1\geq 2$, and recall that $\mathscr K_M^\nabla$ is the interior of the cone generated by the nontrivial simplicial generators of $\myChow^\bullet(M)$.  (It is a subcone of $\mathscr K_M$.)  For any $\ell\in \mathscr K_M^\nabla$, we have $\int_M \ell^d > 0$.  If further $r = d+1 \geq 3$, then the form $Q_\ell^1$ has exactly one positive eigenvalue.  In particular, $\operatorname{HR}_{\mathscr K_M}^{1}$ is satisfied for $M$ with rank $3$.
\end{lem}

\begin{proof}
The statement $\int_M \ell^d > 0$ follows from our dragon Hall-Rado formula (Corollary \ref{cor:volpolH}).  When the rank of $M$ is at least 3, that the form $Q_\ell^1$ has exactly one positive eigenvalue follows from combining Theorem \ref{thm:VPLorentzian} and Proposition \ref{prop:logconcHR}.  Lastly, when rank of $M$ equals 3, the bilinear form $Q_\ell^1$ is the Poincar\'e pairing $A^1(M)\times A^{2-1}(M) \to A^2(M)$, which is non-degenerate and independent of $\ell \in \mathscr K_M^\nabla$.  Thus, we conclude $\operatorname{HR}_{\ell'}^1$ for any element $\ell' \in \mathscr K_M$ since $Q_{\ell'}^1=Q_\ell^1$.
\end{proof}

\begin{proof}[Proof of Theorem \ref{thm:HRone}]
We proceed by induction on the rank of the matroid $M$.  The base case consists of rank $\leq 2$ matroids for $\operatorname{HL}_{\mathscr K_M}^0 \& \operatorname{HR}_{\mathscr K_M}^0$ and rank $\leq 3$ for $\operatorname{HL}_{\mathscr K_M}^1 \&\operatorname{HR}_{\mathscr K_M}^1$.  For these cases, the stated properties are either vacuous or easily verified with \Cref{lem:lorentzianHR}.  Let $M$ now be a loopless matroid of rank $r = d+1$ on a ground set $E$.

Observe that both properties $\operatorname{HL}_{\mathscr K_M}^{ 0}$ and $\operatorname{HR}_{\mathscr K_M}^{ 0}$ hold together if and only if $\int_M \ell^d > 0$ for all $\ell \in \mathscr K_M$. Moreover, given $\operatorname{HL}_\ell^1$, the property $\operatorname{HR}_\ell^1$ holds if and only if $Q_\ell^1$ has exactly one positive eigenvalue.
Combined with Lemma \ref{lem:lorentzianHR} and Proposition \ref{fullHIT}, these facts imply that proving $\operatorname{HL}^{\leq 1}_{\mathscr K_M}$ is sufficient to establish $\operatorname{HR}^{\leq 1}_{\mathscr K_M}$.
 By Remark \ref{rem:combnefeff}, any element $\ell \in \mathscr K_M$ can be written as a non-negative linear combination of $\{x_F \ | \ F\in \mathscr L_M\setminus\{\emptyset,E\}\}$; therefore, by Proposition \ref{localHIT}, to establish  $\operatorname{HL}_{\mathscr K_M}^{\leq 1}$, it suffices in turn to prove $\operatorname{HR}^{\leq 1}$ for $A^\bullet(M)/\operatorname{ann}(x_F)$ for every nonempty proper flat $F$.
 Finally, $A^\bullet(M) / \operatorname{ann}(x_F) \simeq (A(M|F) \otimes A(M/F))^\bullet$ by Lemma \ref{lem:bergmanprod}(2), so by the induction hypothesis and Proposition \ref{HITtens}, the proof is complete.
 \end{proof} 
 
We conclude by recounting the argument in \cite[\S9]{AHK18} that the Hodge-Riemann relations in degree one implies the Heron-Rota-Welsh conjecture.

\begin{lem}\cite[Lemma 9.6]{AHK18}\label{lem:nefineq}
        Let $\ell_1, \ell_2 \in A^1(M)$. If $\ell_2$ is nef, then
        \[
                \left(\int_M \ell_1^2 \ell_2^{d-2} \right) \left( \int_M \ell_2^2 \ell_2^{d-2} \right) \leq \left( \int_M \ell_1\ell_2\ell_2^{d - 2} \right)^2.
        \]
\end{lem}
\begin{proof}
        Suppose first that $\ell_2$ is ample.
        By Theorem \ref{thm:HRone}, $A^\bullet(M)$ satisfies $\mathrm{HL}_{\ell_2}^{\leq 1}$, so we obtain a decomposition $A^1(M) \cong \langle\ell_2\rangle  \oplus P_{\ell_2}^1$ that is orthogonal with respect to the Hodge-Riemann form $Q_{\ell_2}^1$.
        By $\mathrm{HR}_{\ell_2}^{\leq 1}$, $Q_{\ell_2}^1$ is positive definite on $P^1_{\ell_2}$ and negative definite on $\langle \ell_2\rangle$; therefore, the restriction of $Q_{\ell_2}^1$ to the subspace $\langle \ell_1, \ell_2 \rangle \subset A^1(M)$ is neither positive nor negative definite, so
        \[
                \left( \int_M \ell_1^2 \ell_2^{d-2} \right)\left( \int_M \ell_2^2 \ell_2^{d-2} \right) < \left(\int_M \ell_1 \ell_2 \ell_2^{d-2} \right).
        \]
        If $\ell_2$ is merely nef rather than ample, then for any ample element $\ell$, the class $\ell_2(t) := \ell_2 + t \ell$ is ample for all $t > 0$. An ample $\ell$ exists by Lemma \ref{lem:lorentzianHR}.
        Taking a limit as $t \to 0$ in the inequality
        \[
                \left( \int_M \ell_1^2 \ell_2(t)^{d-2} \right)\left( \int_M \ell_2(t)^2 \ell_2(t)^{d-2} \right) < \left(\int_M \ell_1 \ell_2(t) \ell_2(t)^{d-2} \right).
        \]
        yields the desired inequality.
\end{proof}

\begin{cor}\label{cor:logconc}
        For each $0 < k < d$,
        \[
                \mu^{k-1}(M) \mu^{k+1}(M) \leq \mu^k(M)^2.
        \]
\end{cor}
\begin{proof}
        This proof is reproduced from \cite[Proposition 9.8]{AHK18}.
        We proceed by induction on $\rk(M)$. 
        When $k < d-1$, the induction hypothesis applied to the truncation $T_E(M)$ implies the inequality because 
        the absolute values of the lower coefficients of $\overline\chi_{T_E(M)}$ are the same as those of $\overline\chi_M$.
        Now, consider $k = d-1$.
        For any $i \in E$, denote $\alpha  := \sum_{i \in F} x_F \in A^1(M)$ and $\beta := \sum_{i \not \in F} x_F $.
        Both $\alpha$ and $\beta$ are independent of the choice of $i$ and are nef.  Proposition 9.5 of \cite{AHK18} states that $\mu^k(M) = \int_M \alpha^{d-k} \beta^k$; therefore, the desired inequality is
        \[
                \left(\int_M \alpha^2 \beta^{d-2} \right) \left( \int_M \beta^2 \beta^{d-2} \right) \leq \left( \int_M \alpha\beta\beta^{d - 2} \right)^2.
        \]
        Since $\alpha$ and $\beta$ are nef, this inequality holds by Lemma \ref{lem:nefineq}.
\end{proof}


\printbibliography

\end{document}